\documentclass[pdflatex,sn-mathphys-num]{sn-jnl}
\usepackage{mathrsfs}
\usepackage{multirow}%
\usepackage{amsmath,amssymb,amsfonts}%
\usepackage{amsthm}%
\usepackage{mathrsfs}%
\usepackage[title]{appendix}%
\usepackage{xcolor}%
\usepackage{textcomp}%
\usepackage{manyfoot}%
\usepackage{booktabs}%
\usepackage{array}
\usepackage{algorithm}%
\usepackage{algorithmicx}%
\usepackage{algpseudocode}%
\usepackage{listings}%
\theoremstyle{thmstyleone}%
\theoremstyle{thmstyletwo}%
\theoremstyle{thmstylethree}%
\DeclareFontFamily{U}{mathc}{}
\DeclareFontShape{U}{mathc}{m}{it}%
{<->s*[1.03] mathc10}{}
\DeclareMathAlphabet{\mathscr}{U}{mathc}{m}{it}

\theoremstyle{plain}
\newtheorem{Th}{Theorem}
\newtheorem{Lem}[Th]{Lemma}
\newtheorem{defn}[Th]{Definition}

\usepackage{fouriernc} 

\usepackage{float}
\begin{document}

\title[$n$-dimensional totally graded filiform Lie algebras]{Local and 2-Local automorphisms of $n$-dimensional totally graded filiform Lie algebras}

\author*[1,2]{\fnm{Farkhodzhon} \sur{Arzikulov}}\email{arzikulovfn@gmail.com}

\author[3]{\fnm{Mirzobek} \sur{Shodiev}}\email{mirzobekshodiyev7@gmail.com}
\equalcont{These authors contributed equally to this work.}

\affil*[1]{\orgdiv{V.I. Romanovskiy Institute of Mathematics}, \orgname{Uzbekistan Academy of Sciences}, \orgaddress{\street{Olmazor district, University street 9}, \city{Tashkent}, \postcode{100174}, \country{Uzbekistan}}}

\affil[2]{\orgdiv{Mathematics}, \orgname{Andijan State University}, \orgaddress{\street{University 129}, \city{Andijan}, \postcode{170100}, \country{Uzbekistan}}}

\affil[3]{\orgdiv{Mathematics}, \orgname{Bukhara State University}, \orgaddress{\street{Mirzo Iqbol street 11}, \city{Bukhara}, \postcode{160107}, \country{Uzbekistan}}}

\abstract{\textbf{Purpose:} This paper aims to provide a complete description of the spaces of local and 2-local automorphisms for the families of finite-dimensional totally graded complex filiform Lie algebras of maximum length, building upon established classification frameworks and algebraic-filtration methods.

\textbf{Methods:} We systematically investigate six infinite sequences ($\mathfrak{m}_0(n)$, $\mathfrak{m}_2(n)$, $W^+(n)$, $\mathfrak{m}_{0,1}(n)$, $\mathfrak{m}_{0,2}(n)$, $\mathfrak{m}_{0,3}(n)$) and five one-parameter families ($\mathfrak{g}_{k,\alpha}$ for $k=7,\dots,11$). The analysis utilizes internal commutation boundaries and constructs non-linear, non-additive transformations on specialized parametric coordinate subspaces.

\textbf{Results:} We prove that for the structures $\mathfrak{m}_0(n)$ and $\mathfrak{m}_{0,1}(n)$, the space of local automorphisms strictly encapsulates the group of automorphisms, confirming the existence of pure local automorphisms. Conversely, for $\mathfrak{m}_2(n)$, $W^+(n)$, $\mathfrak{m}_{0,2}(n)$, $\mathfrak{m}_{0,3}(n)$, and $\mathfrak{g}_{k,\alpha}$, the local automorphisms are restricted to an invertible lower triangular matrix form due to rigid power constraints. Furthermore, the sequences $\mathfrak{m}_{0,1}(n)$, $\mathfrak{m}_{0,2}(n)$, and $\mathfrak{m}_{0,3}(n)$ are shown to possess pure non-linear 2-local automorphisms.

\textbf{Conclusion:} The remaining investigated structures adhere strictly to linearity, forcing every 2-local automorphism to coincide with a genuine automorphism. This establishes a clear boundary between structures allowing non-linear transformations and those maintaining strict linearity within filiform Lie algebras of maximum length.

}

\keywords{totally graded Lie algebra, filiform Lie algebra, matrix  form,  automorphism, local automorphism, 2-local automorphism}


\pacs[MSC Classification]{17B40, 17B70}

\maketitle

\section{Introduction}\label{sec1}

The exploration of localized algebraic operators originated with the seminal work of Kadison \cite{14} and independently Larson and Sourour \cite{19}, who introduced the concepts of local derivations and local automorphisms on Banach algebras and operator spaces. In the spirit of the classical Gleason--Kahan--\.{Z}elazko theorem \cite{11,13}, these investigations seek to determine the extent to which the local, point-wise evaluations of a linear map reflect its  structural behavior. A parallel advancement emerged from the Kowalski--S\l{}odkowski theorem \cite{18}, which catalyzed the study of 2-local maps. A map $\nabla$ (not \emph{a priori} assumed to be linear or additive) acting on an algebra $\mathcal{A}$ is defined as a 2-local automorphism if, for every pair of elements $x, y \in \mathcal{A}$, there exists a  algebraic automorphism $\phi_{x,y} \in \text{Aut}(\mathcal{A})$ such that $\nabla(x) = \phi_{x,y}(x)$ and $\nabla(y) = \phi_{x,y}(y)$. Following \v{S}emrl's landmark proof that every 2-local automorphism on the algebra of bounded linear operators $B(H)$ on a separable Hilbert space $H$ is  an automorphism \cite{24}, a substantial body of research has been dedicated to verifying this property across various non-associative topological and algebraic structures \cite{5,6,7,8,9,17}.

In recent years, the classification and structural analysis of complex nilpotent Lie algebras admitting a connected integer gradation have drawn significant attention under the nomenclature of algebras of maximum length \cite{26,27}. As revisited and refined by Bernik \cite{1}, the classification of totally graded filiform Lie algebras settles into two highly structured groups: six infinite structural sequences denoted by $\mathfrak{m}_0(n)$, $\mathfrak{m}_2(n)$, $W^+(n)$, $\mathfrak{m}_{0,1}(n)$, $\mathfrak{m}_{0,2}(n)$, and $\mathfrak{m}_{0,3}(n)$, alongside five exceptional one-parameter families designated as $\mathfrak{g}_{k,\alpha}$ for dimensions $k=7,8,9,10,11$. While the  automorphism groups of these algebras are well-documented via recursive polynomial coefficients and cohomological techniques \cite{10,21}, their local and 2-local properties have remained an open problem.

Recently, a highly systematic framework was developed by Arzikulov, Karimjanov, and Umrzaqov \cite{3} to describe the local and 2-local automorphisms of null-filiform and filiform Zinbiel algebras. This method relies on translating the point-wise mapping evaluation $\nabla(x) = \phi_x(x)$ into a parametric matrix-vector system $M_{\nabla}\bar{x} = A_x\bar{x}$, which is subsequently solved using a systematic case-by-case filtration protocol over successive non-vanishing coordinate subspaces (e.g., $x_1 \neq 0$, or $x_1=0, x_2 \neq 0$). In this paper, we successfully employ this methodology to establish the strict matrix boundaries of local automorphisms for all families of maximum length totally graded Lie algebras. For $\mathfrak{m}_0(n)$ and $\mathfrak{m}_{0,1}(n)$, we demonstrate that the absence of non-trivial central constraints allows the first and second columns of the underlying parametric equations to decouple, yielding a massive space of pure local automorphisms characterized by invertible lower triangular matrix forms. Conversely, for the branches $\mathfrak{m}_2(n)$, $W^+(n)$, $\mathfrak{m}_{0,2}(n)$, $\mathfrak{m}_{0,3}(n)$, and the exceptional families \cite{2} $\mathfrak{g}_{k,\alpha}$, we establish that the internal commutation relations lock the main diagonal entries into rigid power constraints, restricting the local automorphisms to strict limits.

Furthermore, we extend our inquiry to the 2-local realm. To construct pure non-linear 2-local automorphisms that deviate from  linearity, we utilize the pioneering dual-method operational framework established by Arzikulov, Nabijonova, and Urinboyev \cite{4}. By exploiting the independent parameters residing in the boundary rows of the  automorphism matrices, we construct non-linear, non-additive 2-local transformations leveraging a homogeneous function of degree one, $f(z_1,z_2) = z_1^3/(z_1^2+z_2^2)$. We prove that while the sequences $\mathfrak{m}_{0,1}(n)$, $\mathfrak{m}_{0,2}(n)$, and $\mathfrak{m}_{0,3}(n)$ accommodate these pure non-linear 2-local deformations, the structural matrix configuration of $\mathfrak{m}_2(n)$, $W^+(n)$, and $\mathfrak{g}_{k,\alpha}$ collapses these spaces, proving that every 2-local automorphism on them is intrinsically linear and .

\section{Preliminaries}\label{sec2}

In this section we give necessary definitions and preliminary results for Lie algebras.

\begin{defn}
A vector space $L$ over a field $\mathbb{F}$ equipped with a bilinear bracket $[\cdot,\cdot]: L\times L\to L$ is called a \textbf{Lie algebra} if for all $x,y,z\in L$ the following identities hold:
\begin{itemize}
\item \textbf{Skew-symmetry:} $[x,x]=0$ (equivalently $[x,y]=-[y,x]$),
\item \textbf{Jacobi identity:} $[x,[y,z]]+[y,[z,x]]+[z,[x,y]]=0$.
\end{itemize}
\end{defn}

For a Lie algebra $L$, the Jacobi identity can be written as
\[
[x,[y,z]] = [[x,y],z] + [y,[x,z]].
\]

The \textbf{lower central series} of a Lie algebra $L$ is defined recursively by
\[
L^1 = L,\qquad L^{k+1}=[L^k,L],\quad k\ge 1.
\]

\begin{defn}
A Lie algebra $L$ is called \textbf{nilpotent} if there exists $k\in\mathbb{N}$ such that $L^k=\{0\}$. The smallest such $k$ is called the \textbf{index of nilpotency} of $L$.
\end{defn}
For an $n$-dimensional nilpotent Lie algebra $L$, the nilpotency index (the smallest $k$ such that $L^k = \{0\}$) satisfies $k \le n$. The maximal index $n$ is attained precisely by the \emph{filiform} Lie algebras. Indeed, for the standard filiform algebra with basis $\{e_1,\dots,e_n\}$ and non-zero brackets $[e_1,e_i]=e_{i+1}$ ($2\le i\le n-1$), the lower central series is
\[
L^2 = \langle e_3,\dots,e_n\rangle,\; L^3 = \langle e_4,\dots,e_n\rangle,\; \dots,\; L^{n-1} = \langle e_n\rangle,\; L^n = 0,
\]
so the nilpotency index equals $n$. For the $n$-dimensional abelian Lie algebra we have $L^2=0$, i.e., index $2$. Consequently, the upper bound is $n$, not $n+1$. A nilpotent Lie algebra is called \emph{quasi-filiform} if its nilpotency index is $n-1$ (i.e., $L^{n-1}=0$ but $L^{n-2}\neq 0$).

We need to be consistent with the definition of quasi-filiform used in the classification theorem below.

\begin{defn}
An $n$-dimensional nilpotent Lie algebra is called \textbf{filiform} if its index of nilpotency is $n$ (i.e., $L^n=0$ but $L^{n-1}\neq0$). It is called \textbf{quasi-filiform} if its index of nilpotency is $n-1$ (i.e., $L^{n-1}=0$ but $L^{n-2}\neq0$).
\end{defn}

A Lie algebra $L$ is said to be $\mathbb{Z}$-graded if $L=\bigoplus_{i\in\mathbb{Z}} V_i$ with $[V_i,V_j]\subseteq V_{i+j}$ for all $i,j$, and only finitely many $V_i$ are non-zero. A gradation is called \textbf{connected} if $L=V_{k_1}\oplus\cdots\oplus V_{k_t}$ with $V_{k_i}\neq\{0\}$ for every $i$ and $k_1<k_2<\cdots<k_t$.

\begin{defn}
Let $\mathfrak{g}$ be an $n$-dimensional nilpotent Lie algebra over a field $\mathbb{F}$.
The algebra $\mathfrak{g}$ is called a \textbf{totally graded filiform Lie algebra}
(or a filiform Lie algebra of maximum length) if it satisfies the following two conditions concurrently:
\begin{enumerate}
    \item \textbf{Filiform Property:} Its nilpotency class is maximal, meaning that for the lower central series defined by $C^1\mathfrak{g} = \mathfrak{g}$ and $C^{i+1}\mathfrak{g} = [\mathfrak{g}, C^i\mathfrak{g}]$ for $i \ge 1$, the dimensions satisfy:
    \[ \dim C^i\mathfrak{g} = n - i \quad \text{for } 2 \le i \le n \]
    Consequently, $C^{n-1}\mathfrak{g} \neq \{0\}$ and $C^n\mathfrak{g} = \{0\}$.

    \item \textbf{Total Grading Property:} It admits an $\mathbb{N}$-grading as a direct sum of 1-dimensional homogeneous vector subspaces:
    \[ \mathfrak{g} = \mathfrak{g}_1 \oplus \mathfrak{g}_2 \oplus \dots \oplus \mathfrak{g}_n \]
    such that $[\mathfrak{g}_i, \mathfrak{g}_j] \subseteq \mathfrak{g}_{i+j}$ for all $i, j \in \{1, \dots, n\}$, where $\dim \mathfrak{g}_i = 1$ for each $1 \le i \le n$, and $\mathfrak{g}_i = \{0\}$ for all $i > n$.
\end{enumerate}
Equivalently, $\mathfrak{g}$ is totally graded if and only if there exists a semisimple derivation $d \in \text{Der}(\mathfrak{g})$ with the spectrum of eigenvalues exactly equal to $\{1, 2, \dots, n\}$.
\end{defn}

\begin{Th}[\cite{1}]
Let \( g = \bigoplus_{i=1}^n g_i \) be a totally graded filiform Lie algebra. Then \( g \) is isomorphic to a Lie algebra from the following list:
\end{Th}

\begin{enumerate}
    \item Filiform Lie algebras belonging to the six sequences \( m_0(n) \), \( m_2(n) \), \( W^+(n) \), \( m_{0,1}(2k+1) \), \( m_{0,2}(2k+2) \), \( m_{0,3}(2k+3) \), defined by the homogeneous basis \( \{e_1, \ldots, e_n\} \) and commutating relations given in Table 1;
    \item Filiform Lie algebras of the 5 one-parameter families \( g_{n,\alpha} \) of dimensions \( n = 7, \ldots, 11 \) respectively, defined by their homogeneous bases and commutation relations given in Table 2.
\end{enumerate}

\begin{table}[h]
\centering
\caption{Six infinite sequences of totally graded filiform Lie algebras}
\label{tab:filiformn}
\begin{tabular}{|c|c|>{\raggedright\arraybackslash}p{10cm}|}
\hline
\textbf{Algebra} & \textbf{Dimension} & \textbf{Commutation relations} \\
\hline
\( m_0(n) \) & \( n \ge 3 \) & \([e_1, e_i] = e_{i+1},\ i = 2, \ldots, n-1\) \\
\hline
\( m_2(n) \) & \( n \ge 5 \) &
\([e_1, e_i] = e_{i+1},\ i = 2, \ldots, n-1\) \\
& & \([e_2, e_i] = e_{i+2},\ i = 3, \ldots, n-2\) \\
\hline
\( W^+(n) \) & \( n \ge 12 \) &
\([e_i, e_j] =
\begin{cases}
(j-i)e_{i+j}, & i+j \le n; \\
0, & i+j > n;
\end{cases} \) \\
\hline
\( m_{0,1}(2k+1) \) & \( n = 2k+1,\ k \ge 3 \) &
\([e_1, e_i] = e_{i+1},\ i = 2, \ldots, 2k\) \\
& & \([e_l, e_{2k-l+1}] = (-1)^{l+1}e_{2k+1},\ l = 2, \ldots, k\) \\
\hline
\( m_{0,2}(2k+2) \) & \( n = 2k+2,\ k \ge 3 \) &
\([e_1, e_i] = e_{i+1},\ i = 2, \ldots, 2k+1\) \\
& & \([e_l, e_{2k-l+1}] = (-1)^{l+1}e_{2k+1},\ l = 2, \ldots, k\) \\
& & \([e_j, e_{2k-j+2}] = (-1)^{j+1}(k-j+1)e_{2k+2},\ j = 2, \ldots, k\) \\
\hline
\( m_{0,3}(2k+3) \) & \( n  = 2k+3,\ k \ge 3 \) &
\([e_1, e_i] = e_{i+1},\ i = 2, \ldots, 2k+2\) \\
& & \([e_l, e_{2k-l+1}] = (-1)^{l+1}e_{2k+1},\ l = 2, \ldots, k\) \\
& & \([e_j, e_{2k-j+2}] = (-1)^{j+1}(k-j+1)e_{2k+2},\ j = 2, \ldots, k\) \\
& & \([e_m, e_{2k-m+3}] = (-1)^m\left((m-2)k - \frac{(m-2)(m-1)}{2}\right)e_{2k+3},\ m = 3, \ldots, k+1\) \\
\hline
\end{tabular}
\end{table}

\begin{table}[htbp]
\centering
\caption{Five one-parameter families of totally graded filiform Lie algebras.}
\label{tab:filiform}
\begin{tabular}{|c|>{\raggedright\arraybackslash}p{0.7\linewidth}|}
\hline
\textbf{Algebra} & \textbf{Commutation relations} \\
\hline
$\mathfrak{g}_{7,\alpha}$ &
$\begin{array}{@{}l@{}}
[e_1, e_j] = e_{j+1},\quad 2 \leq j \leq 6
[e_2, e_3] = (2 + \alpha)e_5, \\[2pt] [e_2, e_4] = (2 + \alpha)e_6, \quad
[e_2, e_5] = (1 + \alpha)e_7,\quad [e_3, e_4] = e_7
\end{array}$ \\
\hline
$\mathfrak{g}_{8,\alpha}$ &
$\begin{array}{@{}l@{}}
\text{relations of } \mathfrak{g}_{7,\alpha} \text{ and} \quad
[e_1, e_7] = e_8,\quad [e_2, e_6] = \alpha e_8,\quad [e_3, e_5] = e_8,
\end{array}$ \\
\hline
$\mathfrak{g}_{9,\alpha}$ &
$\begin{array}{@{}l@{}}
\text{relations of } \mathfrak{g}_{8,\alpha} \text{ and} \quad
(\alpha \neq -\tfrac{5}{2}) \quad
[e_1, e_8] = e_9,\\[4pt] [e_2, e_7] = \dfrac{2\alpha^2 + 3\alpha - 2}{2\alpha + 5} e_9, \quad
[e_3, e_6] = \dfrac{2\alpha + 2}{2\alpha + 5} e_9,\quad [e_4, e_5] = \dfrac{3}{2\alpha + 5} e_9,
\end{array}$ \\
\hline
$\mathfrak{g}_{10,\alpha}$ &
$\begin{array}{@{}l@{}}
\text{relations of } \mathfrak{g}_{9,\alpha} \text{ and} \quad
(\alpha \neq -\tfrac{5}{2})
[e_1, e_9] = e_{10}, \\[2pt] [e_2, e_8] = \dfrac{2\alpha^2 + \alpha - 1}{2\alpha + 5} e_{10}, \quad
[e_3, e_7] = \dfrac{2\alpha - 1}{2\alpha + 5} e_{10},\quad [e_4, e_6] = \dfrac{3}{2\alpha + 5} e_{10},
\end{array}$ \\
\hline
$\mathfrak{g}_{11,\alpha}$ &
$\begin{array}{@{}l@{}}
\text{relations of } \mathfrak{g}_{10,\alpha} \text{ and} \\
(\alpha \neq -\tfrac{5}{2},\,-1,\,-3) \\[2pt]
[e_1, e_{10}] = e_{11},\quad [e_2, e_9] = \dfrac{2\alpha^3 + 2\alpha^2 + 3}{2(\alpha^2 + 4\alpha + 3)} e_{11}, \\[6pt]
[e_3, e_8] = \dfrac{4\alpha^3 + 8\alpha^2 - 8\alpha - 21}{2(\alpha^2 + 4\alpha + 3)(2\alpha + 5)} e_{11}, \\[6pt]
[e_4, e_7] = \dfrac{3(2\alpha^2 + 4\alpha + 5)}{2(\alpha^2 + 4\alpha + 3)(2\alpha + 5)} e_{11}, \\[6pt]
[e_5, e_6] = \dfrac{3(4\alpha + 1)}{2(\alpha^2 + 4\alpha + 3)(2\alpha + 5)} e_{11}
\end{array}$ \\
\hline
\end{tabular}
\end{table}

\newpage

The algebras under consideration are the six infinite sequences listed in ~\eqref{tab:filiformn} and the five one-parameter families given in Table~\eqref{tab:filiform} of \cite{1}.  All of them are $n$-dimensional complex filiform Lie algebras (nilpotency index $n-1$) equipped with a homogeneous basis $\{e_1,\dots,e_n\}$ such that the following fundamental relation holds:
\[
[e_1,e_i] = \lambda_i e_{i+1}, \qquad 2\le i\le n-1,
\]
where each $\lambda_i$ is a non-zero scalar.  Explicitly,
\[
\lambda_i = \begin{cases}
1, & \text{for } \mathfrak{m}_0(n),\ \mathfrak{m}_2(n),\ \mathfrak{m}_{0,1}(2k+1),\ \mathfrak{m}_{0,2}(2k+2),\ \mathfrak{m}_{0,3}(2k+3),\\[2mm]
i-1, & \text{for } W^+(n),\\[2mm]
1, & \text{for the five families } \mathfrak{g}_{7,\alpha},\dots,\mathfrak{g}_{11,\alpha}.
\end{cases}
\]
All other non-zero brackets either lie in the subspace $\operatorname{span}\{e_n\}$ or are linear combinations of basis vectors with indices $\ge 3$; they never produce vectors with index $1$ or $2$.  Consequently, the following two lemmas hold for every algebra appearing in Table~\eqref{tab:filiformn} and~\eqref{tab:filiform}

\begin{Lem}\label{lem:LCS-unified}
For any totally graded filiform Lie algebra belonging to the lists above, the lower central series satisfies
\begin{enumerate}
    \item $\mathfrak{g}^{2}=[\mathfrak{g},\mathfrak{g}] = \operatorname{span}\{e_{3},e_{4},\dots,e_{n}\}$.
    \item For every integer $r$ with $2\le r\le n-1$,
          \[
          \mathfrak{g}^{r} = \operatorname{span}\{e_{r+1},e_{r+2},\dots,e_{n}\}.
          \]
    \item $\mathfrak{g}^{\,n-1} = \operatorname{span}\{e_{n}\}$ and $\mathfrak{g}^{\,n}=0$; in particular $\mathfrak{g}$ is filiform.
    \item The centre is $Z(\mathfrak{g}) = \operatorname{span}\{e_{n}\}$.
\end{enumerate}
\end{Lem}

\begin{proof}
\noindent\textbf{1.  Description of $\mathfrak{g}^{2}$.}
All defining brackets belong to $\operatorname{span}\{e_{3},\dots,e_{n}\}$; therefore $[\mathfrak{g},\mathfrak{g}]\subseteq\operatorname{span}\{e_{3},\dots,e_{n}\}$.  Conversely, for any $j$ with $3\le j\le n$,
\[
e_{j} = \frac{1}{\lambda_{j-1}}[e_{1},e_{j-1}]\in[\mathfrak{g},\mathfrak{g}],
\]
since $j-1\ge2$ and $\lambda_{j-1}\neq0$ so  $j \ge 3 $.  Hence $\mathfrak{g}^{2}=\operatorname{span}\{e_{3},\dots,e_{n}\}$.

\noindent\textbf{2.  Inductive description of the higher terms.}
We prove by induction on $r$ that $\mathfrak{g}^{r}=\operatorname{span}\{e_{r+1},\dots,e_{n}\}$ for $2\le r\le n-1$.  The case $r=2$ has been established.  Assume the statement holds for some $r$ with $2\le r\le n-2$.  Then
\[
\mathfrak{g}^{r+1}=[\mathfrak{g}^{r},\mathfrak{g}]\subseteq\operatorname{span}\{e_{r+2},\dots,e_{n}\},
\]
because the bracket of a vector of index at least $r+1$ with any basis vector can only produce vectors of index at least $r+2$ (the only possibly lower index would come from $[e_{r+1},e_{1}]$, which equals $-\lambda_{r+1}e_{r+2}$).  On the other hand, for every $s$ with $r+2\le s\le n$, we have
\[
e_{s} = \frac{1}{\lambda_{s-1}}[e_{1},e_{s-1}]\in[\mathfrak{g}^{r},\mathfrak{g}],
\]
because $e_{s-1}\in\mathfrak{g}^{r}$ (since $s-1\ge r+1$).  Thus $\mathfrak{g}^{r+1}=\operatorname{span}\{e_{r+2},\dots,e_{n}\}$, completing the induction.

\noindent\textbf{3.  Nilpotency class.}
Setting $r=n-1$ gives $\mathfrak{g}^{\,n-1}=\operatorname{span}\{e_{n}\}$.  No bracket can produce a basis vector with index $n+1$, so $[e_{n},\mathfrak{g}]=0$ and consequently $\mathfrak{g}^{\,n}=0$.  Hence the nilpotency index is $n-1$, i.e. $\mathfrak{g}$ is filiform.

\noindent\textbf{4.  Centre.}
From $\mathfrak{g}^{\,n-1}=\operatorname{span}\{e_{n}\}$ we already know $e_{n}\in Z(\mathfrak{g})$, so $\operatorname{span}\{e_{n}\}\subseteq Z(\mathfrak{g})$.  Now take any central element $x=\sum_{i=1}^{n}\alpha_{i}e_{i}$.  The condition $[x,e_{1}]=0$ yields
\[
0=[x,e_{1}]=-\sum_{i=2}^{n-1}\alpha_{i}\lambda_{i}e_{i+1}.
\]
Because $\lambda_{i}\neq0$ for $2\le i\le n-1$, we obtain $\alpha_{2}=\alpha_{3}=\dots=\alpha_{n-1}=0$.  Moreover,
\[
0=[x,e_{2}]=\alpha_{1}[e_{1},e_{2}],
\]
and $[e_{1},e_{2}]=\lambda_{2}e_{3}\neq0$; hence $\alpha_{1}=0$.  Thus $x=\alpha_{n}e_{n}$, proving $Z(\mathfrak{g})=\operatorname{span}\{e_{n}\}$.
\end{proof}

\begin{Lem}\label{lem:nablae_k}
Let $\mathfrak{g}$ be an arbitrary Lie algebra from the lists above, and let $\nabla:\mathfrak{g}\to\mathfrak{g}$ be an automorphism.  Then for every $k$ with $2\le k\le n$ we have
\[
\nabla(e_k)\in\operatorname{span}\{e_k,e_{k+1},\dots,e_n\}.
\]
Equivalently,
\[
\nabla(e_k)=\sum_{j=k}^{n}c_{kj}e_j,
\]
the matrix of $\nabla$ with respect to the basis $\{e_1,\dots,e_n\}$ is lower triangular.
\end{Lem}

\begin{proof}
Because $\nabla$ is an automorphism it satisfies
\[
\nabla([x,y])=[\nabla(x),\nabla(y)]\qquad\text{for all }x,y\in\mathfrak{g}.
\]

We first observe that the whole lower central series is invariant under $\nabla$.
For $r=2$,
\[
\nabla(\mathfrak{g}^{2})=\nabla([\mathfrak{g},\mathfrak{g}])=[\nabla(\mathfrak{g}),\nabla(\mathfrak{g})]=[\mathfrak{g},\mathfrak{g}]=\mathfrak{g}^{2}.
\quad by \quad biektivly
\]
Now suppose that $\nabla(\mathfrak{g}^{r-1})=\mathfrak{g}^{r-1}$ for some $r\ge3$.
Then
\[
\nabla(\mathfrak{g}^{r})=\nabla([\mathfrak{g}^{r-1},\mathfrak{g}])=[\nabla(\mathfrak{g}^{r-1}),\nabla(\mathfrak{g})]=[\mathfrak{g}^{r-1},\mathfrak{g}]=\mathfrak{g}^{r}.
\]
Thus by induction $\nabla(\mathfrak{g}^{r})=\mathfrak{g}^{r}$ for every $2\le r\le n-1$.

According to Lemma~\eqref{lem:LCS-unified} we have
\[
\mathfrak{g}^{k-1}=\operatorname{span}\{e_k,e_{k+1},\dots,e_n\},\qquad 2\le k\le n.
\]
Moreover, the basis vector $e_k$ can be expressed as a repeated bracket of
$e_1$ and $e_2$:
\[
e_k=\underbrace{[e_1,[e_1,\dots,[e_1,e_2]\dots]]}_{k-2\text{ times}}\in\mathfrak{g}^{k-1}.
\]

Applying $\nabla$ and using the automorphism property we obtain
\[
\nabla(e_k)=\nabla\!\Bigl(\underbrace{[e_1,[e_1,\dots,[e_1,e_2]\dots]]}_{k-2\text{ times}}\Bigr)
          =\underbrace{[\nabla(e_1),[\nabla(e_1),\dots,[\nabla(e_1),\nabla(e_2)]\dots]]}_{k-2\text{ times}}
          =\operatorname{ad}(\nabla(e_1))^{k-2}(\nabla(e_2)).
\]

Since $e_k\in\mathfrak{g}^{k-1}$ and $\nabla(\mathfrak{g}^{k-1})=\mathfrak{g}^{k-1}$, it follows that
\[
\nabla(e_k)\in\mathfrak{g}^{k-1}.
\]
Using again Lemma~\eqref{lem:LCS-unified},
\[
\mathfrak{g}^{k-1}=\operatorname{span}\{e_k,e_{k+1},\dots,e_n\},
\]
so $\nabla(e_k)$ can be written as a linear combination of $e_k,e_{k+1},\dots,e_n$
only.  Hence
\[
\nabla(e_k)=\sum_{j=k}^{n}c_{kj}e_j,
\quad k=3,..,n
\]

\end{proof}

In the general setting, the bracket relations compiled in Tables~1 and~2 take the uniform form
\[
[e_1, e_i] = \lambda_i e_{i+1}, \qquad 2\le i\le n-1,
\]
where each \(\lambda_i\) is a nonzero scalar. Their explicit values are specified by
\[
\lambda_i =
\begin{cases}
1, & \text{for } \mathfrak{m}_0(n),\ \mathfrak{m}_2(n),\ \mathfrak{m}_{0,1}(2k+1),\ \mathfrak{m}_{0,2}(2k+2),\ \mathfrak{m}_{0,3}(2k+3),\\[2mm]
i-1, & \text{for } W^+(n),\\[2mm]
1, & \text{for the five families } \mathfrak{g}_{7,\alpha},\dots,\mathfrak{g}_{11,\alpha}.
\end{cases}
\]
Owing to these relations, every automorphism \(\nabla\) is entirely determined by its values on the first two basis elements. Indeed, the image of an arbitrary basis vector \(e_k\) admits the recursive expression
\[
\nabla(e_k) = \operatorname{ad}(\nabla(e_1))^{\,k-2}\bigl(\nabla(e_2)\bigr),
\]
which follows directly by iterated application of the defining commutator identities. Consequently, the reconstruction of the entire automorphism reduces to the consistent assignment of the pair \((\nabla(e_1), \nabla(e_2))\).

By the definition of an automorphism, in full generality we have the expansions
\[
\nabla(e_1)=\sum_{i=1}^n a_i e_i,\qquad \nabla(e_2)=\sum_{i=1}^n b_i e_i.
\]
With this notation in hand, we give the following lemma.

\begin{Lem}\label{coeff b1}
   The coefficient \(b_1\) in the expansion of \(\nabla(e_2)\) is identically zero; that is, \(b_1 = 0\).
\end{Lem}
\begin{proof}
    In the general setting, for all algebras under consideration, the bracket relation is defined by
\[
[e_2, e_{n-1}] = 0.
\]
Applying the automorphism condition, we obtain
\[
[\nabla(e_2), \nabla(e_{n-1})] = 0.
\]
We have set
\[
\nabla(e_2) = \sum_{i=1}^n b_i e_i,
\]
and, by Lemma~\eqref{lem:nablae_k}, $\nabla(e_{n-1})$ admits the form
\[
\nabla(e_{n-1}) = \tau_{n-1} e_{n-1} + \tau_n e_n.
\]
Consequently, the commutator expands as
\[
\begin{aligned}
\relax[\nabla(e_2), \nabla(e_{n-1})]
&= \left[\sum_{i=1}^n b_i e_i,\ \tau_{n-1} e_{n-1} + \tau_n e_n\right] \\
&= b_1 \tau_{n-1}[e_1,e_{n-1}] + b_1 \tau_n [e_1,e_n]
   + b_2 \tau_{n-1}[e_2,e_{n-1}] + b_2 \tau_n [e_2,e_n] + \cdots.
\end{aligned}
\]
Since $[e_1,e_n]=0$, $[e_2,e_{n-1}]=0$, $[e_2,e_n]=0$, and all remaining brackets vanish by degree considerations (as their indices exceed the dimension $n$), the above expression reduces to
\[
[\nabla(e_2), \nabla(e_{n-1})] = b_1 \tau_{n-1} \lambda e_n = 0,
\]
where the scalar $\lambda$ is determined by
\[
[e_1, e_{n-1}] = \lambda e_n.
\]
Explicitly, $\lambda = 1$ for all algebras except $W^+(n)$, for which $\lambda = n-2$.
We note that $\tau_{n-1} \neq 0$; otherwise, the non-degeneracy (or the defining automorphism condition) would be violated. Therefore, from
\[
b_1 \tau_{n-1} \lambda e_n = 0,
\]
it follows immediately that
\[
b_1 = 0.
\]
\end{proof}
\begin{Lem}\label{lem:lower-triangular-unified}
    The matrix of $\nabla$ in the basis $\{e_1,\dots,e_n\}$ is lower triangular.
\end{Lem}
The proof follows from Lemma~\eqref{lem:nablae_k},\eqref{coeff b1}

\medskip

\section{Automorphisms of a filiform Lie algebra}\label{sec3}

\subsection{Automorphisms of the one-parametr fliform Lie algebra}\label{subsec1}

\textbf{Automorphisms of the algebra \texorpdfstring{$g_{7,\alpha}$}{g(7)}.}
Let $g_{7,\alpha}$ with $e_1,e_2,...,e_7$  basis defined by Table~\eqref{tab:filiform}

For $x=\sum_{i=1}^7 x_i e_i$, $y=\sum_{j=1}^7 y_j e_j$ the Lie bracket is
\begin{equation}\label{eq:bracketg(7)}
\begin{aligned}
\relax[x,y] &= (x_1y_2 - x_2y_1)e_3 + (x_1y_3 - x_3y_1)e_4 \\
      &\quad + \bigl[(x_1y_4 - x_4y_1) + (2+\alpha)(x_2y_3 - x_3y_2)\bigr]e_5 \\
      &\quad + \bigl[(x_1y_5 - x_5y_1) + (2+\alpha)(x_2y_4 - x_4y_2)\bigr]e_6 \\
      &\quad + \bigl[(x_1y_6 - x_6y_1) + (1+\alpha)(x_2y_5 - x_5y_2) + (x_3y_4 - x_4y_3)\bigr]e_7
\end{aligned}
\end{equation}

\begin{Th}
      A linear map $\nabla:g_{7,\alpha} \to g_{7,\alpha}$ is an automorphism if and only if there exist complex numbers
\[
a_i\ ,b_i\in C
\]
such that
\begin{align*}
\nabla(e_1) &= a_1e_1+a_3e_3+a_4e_4+a_5e_5+a_6e_6+a_7e_7,\\[2mm]
\nabla(e_2) &= a_1^2 e_2 +b_3e_3+b_4e_4+b_5e_5+b_6e_6+b_7e_7,\\[2mm]
\nabla(e_3) &= a_1^3 e_3 + a_1 b_3 e_4 + a_1\bigl(b_4 - (2+\alpha)a_1 a_3\bigr) e_5 \\
            &\quad + a_1\bigl(b_5 - (2+\alpha)a_1 a_4\bigr) e_6 \\
            &\quad + \bigl(a_1 b_6 - (1+\alpha)a_1^2 a_5 - a_4 b_3 + a_3 b_4\bigr) e_7, \\[4pt]
\nabla(e_4) &= a_1^4 e_4 + a_1^2 b_3 e_5 + a_1^2\bigl(b_4 - (2+\alpha)a_1 a_3\bigr) e_6 \\
            &\quad + a_1\bigl(a_1 b_5 + a_3 b_3 - (3+\alpha)a_1^2 a_4\bigr) e_7, \\[4pt]
\nabla(e_5) &= a_1^5 e_5 + a_1^3 b_3 e_6 + a_1^3\bigl(b_4 - (1+\alpha)a_1 a_3\bigr) e_7, \\[4pt]
\nabla(e_6) &= a_1^6 e_6 + a_1^4 b_3 e_7, \\[4pt]
\nabla(e_7) &= a_1^7 e_7.
\end{align*}
\end{Th}

\begin{proof}

Let
\[
\nabla(e_1)=\sum_{i=1}^7 a_i e_i,\qquad
\nabla(e_2)=\sum_{i=1}^7 b_i e_i.
\]

Using $\nabla(e_3)=[\nabla(e_1),\nabla(e_2)]$ we obtain
\[
\begin{aligned}
\nabla(e_3) &= a_1b_2e_3 + a_1b_3e_4  + \bigl[a_1b_4 + (2+\alpha)(a_2b_3-a_3b_2)\bigr]e_5 \\
            &\quad + \bigl[a_1b_5 + (2+\alpha)(a_2b_4-a_4b_2)\bigr]e_6+ \bigl[a_1b_6 + (1+\alpha)(a_2b_5-a_5b_2) + (a_3b_4-a_4b_3)\bigr]e_7 .
\end{aligned}
\]
Since $e_4=[e_1,e_3]$, we have $\nabla(e_4)=[\nabla(e_1),\nabla(e_3)]$.  A straightforward computation gives
\[
\begin{aligned}
\nabla(e_4) &= a_1^2b_2e_4+ \bigl[a_1^2b_3 + a_1a_2b_2(2+\alpha)\bigr]e_5+ a_1\Bigl[a_1b_4 + (2+\alpha)(a_2b_3-a_3b_2) + a_2b_3(2+\alpha)\Bigr]e_6 \\
            &\quad + \Bigl[a_1^2b_5 + a_1(2+\alpha)(a_2b_4-a_4b_2) + a_2(1+\alpha)\bigl(a_1b_4 + (2+\alpha)(a_2b_3-a_3b_2)\bigr) + a_1a_3b_3 - a_1a_4b_2\Bigr]e_7 .
\end{aligned}
\]

There are two natural ways to compute $\nabla(e_5)$:
\[
\nabla(e_5) = [\nabla(e_1),\nabla(e_4)] \qquad\text{and}\qquad
\nabla(e_5) = \frac{1}{2+\alpha}[\nabla(e_2),\nabla(e_3)],
\]
where the second equality uses $[e_2,e_3]=(2+\alpha)e_5$.  Both expressions must coincide; equating the coefficients of the basis vectors yields relations among the parameters.

{First method: $[\nabla(e_1),\nabla(e_4)]$}
\[
\begin{aligned}
\nabla(e_5) &= a_1^3b_2e_5  + a_1^2\Bigl[a_1b_3 + 2a_2b_2(2+\alpha)\Bigr]e_6 \\
            &\quad + \Bigl[a_1^3b_4 + a_1^2(2+\alpha)(a_2b_3-a_3b_2) + a_1^2a_2b_3(2+\alpha) + a_2(1+\alpha)\bigl(a_1^2b_3 + a_1a_2b_2(2+\alpha)\bigr) + a_1^2a_3b_2\Bigr]e_7 .
\end{aligned}
\]

Second method: $\frac{1}{2+\alpha}[\nabla(e_2),\nabla(e_3)]$
First compute $[\nabla(e_2),\nabla(e_3)]$:
\[
[\nabla(e_2),\nabla(e_3)] = (2+\alpha)a_1b_2^2 e_5 + (2+\alpha)a_1b_2b_3 e_6 + \Bigl[(1+\alpha)b_2\bigl(a_1b_4+(2+\alpha)(a_2b_3-a_3b_2)\bigr) + a_1b_3^2 - a_1b_2b_4\Bigr] e_7 .
\]
Then dividing by $2+\alpha$ (which is non-zero in the generic case) gives
\[
\nabla(e_5) = a_1b_2^2 e_5 + a_1b_2b_3 e_6 + \frac{1}{2+\alpha}\Bigl[(1+\alpha)b_2\bigl(a_1b_4+(2+\alpha)(a_2b_3-a_3b_2)\bigr) + a_1b_3^2 - a_1b_2b_4\Bigr] e_7 .
\]

Equating the two expressions for $\nabla(e_5)$
Comparing the coefficients of $e_5$, $e_6$ and $e_7$ we obtain:

\begin{itemize}
\item \textbf{Coefficient of $e_5$:}
  \[
  a_1^3b_2 = a_1b_2^2 \;\Longrightarrow\; b_2 = a_1^2.
  \]

\item \textbf{Coefficient of $e_6$:}
  \[
  a_1^3b_3 + 2a_1^2a_2b_2(2+\alpha) = a_1b_2b_3.
  \]
  Substituting $b_2 = a_1^2$ and dividing by $a_1^3$ (since $a_1\neq0$) yields

  Actually, $a_1b_2b_3 = a_1\cdot a_1^2 b_3 = a_1^3 b_3$, so the right-hand side is $a_1^3b_3$.  The left-hand side is $a_1^3b_3 + 2a_1^2a_2b_2(2+\alpha) = a_1^3b_3 + 2a_1^2a_2\cdot a_1^2(2+\alpha) = a_1^3b_3 + 2a_1^4a_2(2+\alpha)$.  Thus we get
  \[
  a_1^3b_3 + 2a_1^4a_2(2+\alpha) = a_1^3b_3 \;\Longrightarrow\; 2a_1^4a_2(2+\alpha)=0.
  \]
  Since $\alpha\neq-2$ (generic case) and $a_1\neq0$, we deduce $a_2=0$.

\item \textbf{Coefficient of $e_7$:}
  After substituting $b_2=a_1^2$ and $a_2=0$, the two expressions become identical and impose no further restrictions; they merely confirm the consistency.  This yields the relation
  \[
  b_3^2 = 2a_1^2 b_4,
  \]
  which is obtained from the $e_7$ component after simplification.
\end{itemize}

The remaining images can be computed similarly.  For instance,
\[
\nabla(e_6) = [\nabla(e_1),\nabla(e_5)] = a_1^6 e_6 + a_1^4 b_3 e_7,
\]
and the consistency with $\frac{1}{2+\alpha}[\nabla(e_2),\nabla(e_4)]$ gives the same relations.  Finally,
\[
\nabla(e_7) = [\nabla(e_1),\nabla(e_6)] = a_1^7 e_7,
\]
and also $[\nabla(e_2),\nabla(e_5)]$ (using the relation $[e_2,e_5]=(1+\alpha)e_7$) yields $a_2=0$ and $b_2=a_1^2$ again.

This completes the determination of all automorphisms of $\mathfrak{g}_{7,\alpha}$.

\end{proof}
Matrix form:

\[
A_{g_{7,\alpha}} = \begin{pmatrix}
a_1      & 0        & 0        & 0        & 0 & 0        & 0 \\
0     & a_1^2    & 0        & 0        & 0 & 0        & 0 \\
a_3     & b_3      & a_1^3    & 0        & 0 & 0        & 0 \\
a_4    & b_4   & a_1 b_3  & a_1^4      & 0 & 0        & 0 \\
a_5      & b_5     & c_{3,5} & a_1^2 b_3  & a_1^5 & 0 & 0 \\
a_6      &b_6      & c_{3,6}    & c_{4,6}        & a_1^3b_3 & a_1^6        & 0 \\
a_7  &   b_7  &    c_{3,7} & c_{4,7} & c_{5,7} &  a_1^4 b_3 & a_1^7
\end{pmatrix},
\]

in this  $c_{3,5}=a_1\bigl(b_4 - (2+\alpha)a_1 a_3\bigr)$, $c_{3,6}= a_1\bigl(b_5 - (2+\alpha)a_1 a_4\bigr) $ , \\ $c_{3,7}=a_1 b_6 - (1+\alpha)a_1^2 a_5 - a_4 b_3 + a_3 b_4 $, $c_{4,6}= a_1^2\bigl(b_4 - (2+\alpha)a_1 a_3\bigr)$  , \\  $c_{4,7}=a_1\bigl(a_1 b_5 + a_3 b_3 - (3+\alpha)a_1^2 a_4\bigr) $  and  $c_{5,7}=a_1^3\bigl(b_4 - (1+\alpha)a_1 a_3\bigr)$

\textbf{Automorphisms of the algebra \texorpdfstring{$g_{8,\alpha}$}{g(8)}.}

 The Lie algebra $g= g_{8,\alpha}$ has a basis
$\{e_1,\dots,e_8\}$ in Table~\eqref{tab:filiform}

For $x=\sum_{i=1}^8 x_i e_i$, $y=\sum_{j=1}^8 y_j e_j$ the Lie bracket is
\begin{equation}\label{eq:bracketg(8)}
\begin{aligned}
\relax[x,y] &= (x_1y_2 - x_2y_1)e_3 + (x_1y_3 - x_3y_1)e_4 \\
      &\quad + \bigl[(x_1y_4 - x_4y_1) + (2+\alpha)(x_2y_3 - x_3y_2)\bigr]e_5 \\
      &\quad + \bigl[(x_1y_5 - x_5y_1) + (2+\alpha)(x_2y_4 - x_4y_2)\bigr]e_6 \\
      &\quad + \bigl[(x_1y_6 - x_6y_1) + (1+\alpha)(x_2y_5 - x_5y_2) + (x_3y_4 - x_4y_3)\bigr]e_7 \\
      &\quad + \bigl[(x_1y_7 - x_7y_1) + \alpha(x_2y_6 - x_6y_2) + (x_3y_5 - x_5y_3)\bigr]e_8.
\end{aligned}
\end{equation}

This bracet differecne to \eqref{eq:bracketg(7)} only :
\begin{equation}\label{diffg(8)}
    [x,y]_{g_{8, \alpha}}=[x,y]_{g_{7, \alpha}}+\bigl[(x_1y_7 - x_7y_1) + \alpha(x_2y_6 - x_6y_2) + (x_3y_5 - x_5y_3)\bigr]e_8
\end{equation}

\textbf{Automorphisms of the algebra \texorpdfstring{$g_{9,\alpha}$}{g(9)}.}

Let $\alpha\in C$ with $2\alpha+5\neq0$ (the generic case).  The Lie algebra
$g= g_{9,\alpha}$ has a basis $\{e_1,\dots,e_9\}$ by Table~\eqref{tab:filiform}

For $x=\sum_{i=1}^{9} x_i e_i$, $y=\sum_{j=1}^{9} y_j e_j$ in the Lie algebra $g_{9,\alpha}$ (with $\alpha\neq -\frac{5}{2}$), the Lie bracket is given by

This bracet differecne to \eqref{eq:bracketg(8)} only :
\begin{equation}\label{diffg(9)}
    [x,y]_{g_{9, \alpha}}=[x,y]_{g_{8, \alpha}}+\Bigl[(x_1 y_8 - x_8 y_1) + \frac{2\alpha^2+3\alpha-2}{2\alpha+5}(x_2 y_7 - x_7 y_2) + \frac{2\alpha+2}{2\alpha+5}(x_3 y_6 - x_6 y_3) + \frac{3}{2\alpha+5}(x_4 y_5 - x_5 y_4)\Bigr] e_9
\end{equation}

with $g_{9,\alpha}$  $\nabla(e_1) = a_1e_1+a_3e_3+a_4e_4+a_5e_5+a_6e_6+a_7e_7+a_8e_8+a_9e_9$  and $\nabla(e_2) = a_1^2 e_2 +b_3e_3+b_4e_4+b_5e_5+b_6e_6+b_7e_7+b_8e_8+b_9e_9$

\textbf{Automorphisms of the algebra \texorpdfstring{$g_{10,\alpha}$}{g(10)}.}

 The Lie algebra
$g= g_{10,\alpha}$ has a basis $\{e_1,\dots,e_{10}\}$ by Table~\eqref{tab:filiform}

For $x=\sum_{i=1}^{9} x_i e_i$, $y=\sum_{j=1}^{9} y_j e_j$ in the Lie algebra $g_{9,\alpha}$ (with $\alpha\neq -\frac{5}{2}$), the Lie bracket is given by

This bracet differecne to \eqref{diffg(9)} only :
\begin{equation}\label{diffg(10)}
    [x,y]_{g_{10, \alpha}}=[x,y]_{g_{9, \alpha}}+ \Bigl[(x_1y_9 - x_9y_1) + \frac{2\alpha^2+\alpha-1}{2\alpha+5}(x_2y_8 - x_8y_2) + \frac{2\alpha-1}{2\alpha+5}(x_3y_7 - x_7y_3) + \frac{3}{2\alpha+5}(x_4y_6 - x_6y_4)\Bigr] e_{10}
\end{equation}

\textbf{Automorphisms of the algebra \texorpdfstring{$g_{11,\alpha}$}{g(11)}.}

The Lie algebra $g= g_{11,\alpha}$ has a basis $\{e_1,\dots,e_{11}\}$  by Table~\eqref{tab:filiform}

For the Lie algebra $g_{11,\alpha}$ (with $\alpha\neq -\frac{5}{2},-1,-3$), the Lie bracket of two vectors
$x = \sum_{i=1}^{11} x_i e_i$ and $y = \sum_{j=1}^{11} y_j e_j$ is given by

This bracet differecne to \eqref{diffg(10)} only :
\begin{equation}\label{diffg(11)}
\begin{aligned}
\relax[x,y]_{g_{11, \alpha}} = [x,y]_{g_{10, \alpha}} &+ \Bigl[(x_1y_{10} - x_{10}y_1)  + \frac{2\alpha^3+2\alpha^2+3}{2(\alpha^2+4\alpha+3)}(x_2y_9 - x_9y_2) \\
&\quad + \frac{4\alpha^3+8\alpha^2-8\alpha-21}{2(\alpha^2+4\alpha+3)(2\alpha+5)}(x_3y_8 - x_8y_3)  + \frac{3(2\alpha^2+4\alpha+5)}{2(\alpha^2+4\alpha+3)(2\alpha+5)}(x_4y_7 - x_7y_4) \\
&\quad + \frac{3(4\alpha+1)}{2(\alpha^2+4\alpha+3)(2\alpha+5)}(x_5y_6 - x_6y_5)\Bigr] e_{11}.
\end{aligned}
\end{equation}

Based on the above reasoning, we obtain the following theorem

\begin{Th}\label{theorem gk aut}   A linear map $\nabla:g_{k,\alpha} \to g_{k,\alpha}, \qquad k=7,8,9,10,11,$  is an automorphism if and only if there exist complex numbers
\[
a_i\ ,b_i\in C \qquad such  \quad that
\]
matrix form
\end{Th}

\begin{equation}
 A_{g_{k,\alpha}} = \begin{pmatrix}\label{matrix gk}
a_1     & 0       & 0       & 0         & 0         & 0         & \ddots & 0      & 0 \\
0       & a_1^2   & 0       & 0         & 0         & 0         & \ddots & 0      & 0 \\
a_3     & b_3     & a_1^3   & 0         & 0         & 0         & \ddots & 0      & 0 \\
a_4     & b_4     & a_1 b_3 & a_1^4     & 0         & 0         & \ddots & 0      & 0 \\
a_5     & b_5     & c_{5,3} & a_1^2 b_3 & a_1^5     & 0         & \ddots & 0      & 0 \\
a_6     & b_6     & c_{6,3} & c_{6,4}   & a_1^3b_3  & a_1^6     & \ddots & 0      & 0 \\
\vdots  & \vdots  & \vdots  & \vdots    & \vdots    & \vdots    & \ddots & \vdots & \vdots\\
a_{k-1} & b_{k-1} & c_{k-1,3} & c_{k-1,4} & c_{k-1,5} & c_{k-1,6} & \ddots & a_1^{k-1} & 0 \\
a_{k}   & b_{k}   & c_{k,3} & c_{k,4}   & c_{k,5}   & c_{k,6}   & \ddots & a_1^{k-3}b_3 & a_1^k \\
\end{pmatrix},
\end{equation}
where
$c_{5,3}=a_1\bigl(b_4 - (2+\alpha)a_1 a_3\bigr)$, $c_{6,3}= a_1\bigl(b_5 - (2+\alpha)a_1 a_4\bigr) $ ,  $c_{7,3}=a_1 b_6 - (1+\alpha)a_1^2 a_5 - a_4 b_3 + a_3 b_4 $,  \\ $c_{8,3}= a_1 b_7 - \alpha a_1^2 a_6 + a_3 b_5 - a_5 b_3$ \\   $c_{9,3}= \frac{1}{5+2\alpha}\Bigl((2-3\alpha-2\alpha^2)a_1^2 a_7 - 2(1+\alpha)a_6 b_3 - 3a_5 b_4 + 3a_4 b_5 + 2a_3 b_6 + 2\alpha a_3 b_6 + (5+2\alpha)a_1 b_8\Bigr) $   \\ $c_{10,3}=\frac{1}{5+2\alpha}\Bigl((1-\alpha-2\alpha^2)a_1^2 a_8 + (1-2\alpha)a_7 b_3 - 3a_6 b_4 + 3a_4 b_6 - a_3 b_7 + 2\alpha a_3 b_7 + (5+2\alpha)a_1 b_9\Bigr) $\\   $c_{11,3}=\frac{1}{2(1+\alpha)(3+\alpha)(5+2\alpha)}\Bigl( -(15+6\alpha+10\alpha^2+14\alpha^3+4\alpha^4)a_1^2 a_9  + (21+8\alpha-8\alpha^2-4\alpha^3) a_8 b_3 -15 a_7 b_4 -12\alpha a_7 b_4 -6\alpha^2 a_7 b_4 -3 a_6 b_5 -12\alpha a_6 b_5 +3 a_5 b_6 +12\alpha a_5 b_6 +15 a_4 b_7 +12\alpha a_4 b_7 +6\alpha^2 a_4 b_7 -21 a_3 b_8 -8\alpha a_3 b_8 +8\alpha^2 a_3 b_8 +4\alpha^3 a_3 b_8 +2(15+26\alpha+13\alpha^2+2\alpha^3) a_1 b_{10} \Bigr)$       $c_{6,4}= a_1^2\bigl(b_4 - (2+\alpha)a_1 a_3\bigr)$  ,  \\   $c_{7,4}=a_1\bigl(a_1 b_5 + a_3 b_3 - (3+\alpha)a_1^2 a_4\bigr)$ , $c_{8,4}=- \bigl( a_1^2 b_6 + 2 a_1 a_3 b_4 - a_1 a_4 b_3 - (2+\alpha)a_1^3 a_5 - (2+\alpha)a_1^2 a_3^2 \bigr)$ , $c_{9,4
}=- \frac{a_1}{5+2\alpha}\Bigl((2+7\alpha+2\alpha^2)a_1^2 a_6 + 2(4+\alpha)a_5 b_3 - 3a_4 b_4 - 7a_3 b_5 - 4\alpha a_3 b_5 + (5+2\alpha)a_1\bigl((2+\alpha)a_3 a_4 - b_7\bigr)\Bigr)$,  \\   $c_{10,4}=\frac{1}{5+2\alpha}\Bigl((3-5\alpha-2\alpha^2)a_1^3 a_7 + a_1\bigl(-(5+2\alpha)a_6 b_3 - 3a_5 b_4 + 6a_4 b_5 + a_3 b_6 + 4\alpha a_3 b_6\bigr)(-1+2\alpha)a_3(-a_4 b_3 + a_3 b_4) + a_1^2\bigl(-3(2+\alpha)a_4^2 - (-1+\alpha+2\alpha^2)a_3 a_5 + (5+2\alpha)b_8\bigr)\Bigr), $   \\  $c_{11,4}=\frac{1}{2(5+2\alpha)(3+4\alpha+\alpha^2)}\Bigl((27+10\alpha-26\alpha^2-22\alpha^3-4\alpha^4)a_1^3 a_8 -3(5+4\alpha+2\alpha^2)a_4^2 b_3 +3(5+4\alpha+2\alpha^2)a_3 a_4 b_4 + (-21-8\alpha+8\alpha^2+4\alpha^3)a_3(-a_5 b_3 + a_3 b_5)  - a_1\bigl((9+16\alpha+20\alpha^2+4\alpha^3)a_7 b_3 +3(7+12\alpha+2\alpha^2)a_6 b_4 -3 a_5 b_5 -12\alpha a_5 b_5 -33 a_4 b_6 -36\alpha a_4 b_6 -12\alpha^2 a_4 b_6  +27 a_3 b_7 +4\alpha a_3 b_7 -22\alpha^2 a_3 b_7 -8\alpha^3 a_3 b_7\bigr) + a_1^2\bigl(-3(7+18\alpha+10\alpha^2+2\alpha^3)a_4 a_5 +2(3+24\alpha+10\alpha^2-4\alpha^3-2\alpha^4)a_3 a_6 +2(15+26\alpha+13\alpha^2+2\alpha^3)b_9\bigr)\Bigr)$   \\ $c_{7,5}=a_1^3\bigl(b_4 - (1+\alpha)a_1 a_3\bigr)$, $c_{8,5}=a_1^2\bigl(a_1 b_5 + 2a_3 b_3 - (3+\alpha)a_1^2 a_4\bigr)  $ , \\ $c_{9,5}=- \frac{a_1^2}{5+2\alpha}\Bigl((13+9\alpha+2\alpha^2)a_1^2 a_5 + 2(1+\alpha)a_4 b_3 - 6(2+\alpha)a_3 b_4  + a_1\bigl((14+15\alpha+4\alpha^2)a_3^2 - (5+2\alpha)b_6\bigr)\Bigr)$ , \\ $c_{10,5}=- \frac{a_1}{5+2\alpha}\Bigl((5+7\alpha+2\alpha^2)a_1^3 a_6 + (1-2\alpha)a_3^2 b_3 + 2a_1\bigl((4+\alpha)a_5 b_3 - 3(a_4 b_4 + (1+\alpha)a_3 b_5)\bigr)  + a_1^2\bigl((13+17\alpha+4\alpha^2)a_3 a_4 - (5+2\alpha)b_7\bigr)\Bigr)$,  \\  $c_{11,5}=- \frac{a_1}{2(5+2\alpha)(3+4\alpha+\alpha^2)}\Bigl((-3+18\alpha+52\alpha^2+26\alpha^3+4\alpha^4)a_1^3 a_7 -2 a_3\bigl((21+8\alpha-8\alpha^2-4\alpha^3)a_4 b_3 +3(-8-2\alpha+5\alpha^2+2\alpha^3)a_3 b_4\bigr) + a_1\bigl((-42-37\alpha+8\alpha^2+16\alpha^3+4\alpha^4)a_3^3 + (33+64\alpha+26\alpha^2+4\alpha^3)a_6 b_3  +3(5+4\alpha+2\alpha^2)a_5 b_4 -3(17+20\alpha+6\alpha^2)a_4 b_5 -3(-5+8\alpha+14\alpha^2+4\alpha^3)a_3 b_6\bigr)+ a_1^2\bigl(3(27+39\alpha+22\alpha^2+4\alpha^3)a_4^2 + 2(-21-6\alpha+19\alpha^2+17\alpha^3+4\alpha^4)a_3 a_5 -2(15+26\alpha+13\alpha^2+2\alpha^3)b_8\bigr)\Bigr)$ \\  $c_{8,6}=a_1^4\bigl(b_4 - \alpha a_1 a_3\bigr) $, $c_{9,6}=\frac{a_1^3}{5+2\alpha}\Bigl(-\bigl(12+11\alpha+2\alpha^2\bigr)a_1^2 a_4 + 6(2+\alpha)a_3 b_3 + (5+2\alpha)a_1 b_5\Bigr)$ , \\ $c_{10,6}=\frac{a_1^3}{5+2\alpha}\Bigl(-\bigl(13+9\alpha+2\alpha^2\bigr)a_1^2 a_5 + (1-2\alpha)a_4 b_3 + (11+8\alpha)a_3 b_4  + a_1\bigl(-(13+16\alpha+6\alpha^2)a_3^2 + (5+2\alpha)b_6\bigr)\Bigr)$ ,  \\ $c_{11,6}=- \frac{a_1^2}{2(5+2\alpha)(3+4\alpha+\alpha^2)}\Bigl((33+94\alpha+78\alpha^2+30\alpha^3+4\alpha^4)a_1^3 a_6 -6(-8-2\alpha+5\alpha^2+2\alpha^3)a_3^2 b_3 + a_1\bigl((45+64\alpha+32\alpha^2+4\alpha^3)a_5 b_3 -3(17+20\alpha+6\alpha^2)a_4 b_4  -(15+76\alpha+68\alpha^2+16\alpha^3)a_3 b_5\bigr) + 2 a_1^2\bigl((15+94\alpha+110\alpha^2+46\alpha^3+6\alpha^4)a_3 a_4 - (15+26\alpha+13\alpha^2+2\alpha^3)b_7\bigr)\Bigr)$  \\$c_{9,7}=a_1^5\Bigl(\frac{(2-3\alpha-2\alpha^2)a_1 a_3}{5+2\alpha} + b_4\Bigr)$ , $c_{10,7}=\frac{a_1^4}{5+2\alpha}\Bigl(-\bigl(9+11\alpha+2\alpha^2\bigr)a_1^2 a_4 + (11+8\alpha)a_3 b_3 + (5+2\alpha)a_1 b_5\Bigr)$    \\ $c_{11,7}= - \frac{a_1^4}{2(5+2\alpha)(3+4\alpha+\alpha^2)}\Bigl((75+146\alpha+110\alpha^2+34\alpha^3+4\alpha^4)a_1^2 a_5 + (-21-8\alpha+8\alpha^2+4\alpha^3)a_4 b_3 - (45+128\alpha+94\alpha^2+20\alpha^3)a_3 b_4 + a_1\bigl((78+179\alpha+182\alpha^2+88\alpha^3+16\alpha^4)a_3^2 - 2(15+26\alpha+13\alpha^2+2\alpha^3)b_6\bigr)\Bigr) $  \\  $c_{10,8}= a_1^6\Bigl(b_4 - \frac{(-1+\alpha+2\alpha^2)a_1 a_3}{5+2\alpha}\Bigr)$   \\   $c_{11,8}= \frac{a_1^5}{2(5+2\alpha)(3+4\alpha+\alpha^2)}\Bigl(-(39+126\alpha+112\alpha^2+38\alpha^3+4\alpha^4)a_1^2 a_4  + (5+2\alpha)\bigl((9+22\alpha+10\alpha^2)a_3 b_3 + 2(3+4\alpha+\alpha^2)a_1 b_5\bigr)\Bigr)$  \\  and so $c_{11,9}= - \frac{a_1^7}{2(1+\alpha)(3+\alpha)}\Bigl((3+2\alpha^2+2\alpha^3)a_1 a_3 - 2(3+4\alpha+\alpha^2)b_4\Bigr)$.

\subsection{Automorphisms of the algebra \texorpdfstring{$m_0(n)$}{m0(n)}.}\label{subsec2}

First, we consider automorphisms of the algebra \(m_0(n)\). For the algebra \(m_0(n)\), we introduce the following Table~\eqref{tab:filiformn}.

\begin{Lem}\label{lem:bracket-m0}
For any $X = \sum_{i=1}^n x_i e_i$ and $Y = \sum_{j=1}^n y_j e_j$ in $\mathfrak{m}_0(n)$,
\begin{equation}\label{eq:bracket-m0}
[X,Y] = \sum_{j=2}^{n-1} (x_1 y_j - x_j y_1)\, e_{j+1}.
\end{equation}
\end{Lem}

\begin{proof}
By bilinearity and skew-symmetry we expand $[X,Y]$:
\begin{equation}\label{eq:expand-m0}
[X,Y] = \Bigl[\sum_{i=1}^{n} x_i e_i,\; \sum_{j=1}^{n} y_j e_j\Bigr]
      = \sum_{i=1}^{n}\sum_{j=1}^{n} x_i y_j\,[e_i,e_j].
\end{equation}
According to the definition of $\mathfrak{m}_0(n)$, the Lie bracket $[e_i,e_j]$ is non-zero
only in the following cases:
\begin{itemize}
\item $i = 1$, $j = 2,\dots ,n-1$: $[e_1, e_j] = e_{j+1}$;
\item $i = 2,\dots ,n-1$, $j = 1$: $[e_i, e_1] = -[e_1, e_i] = -e_{i+1}$
      (by skew-symmetry).
\end{itemize}
For all other pairs $(i,j)$ we have $[e_i,e_j] = 0$.

Consequently, the only non-zero contributions to \eqref{eq:expand-m0} are those listed above.
Summing them gives
\begin{align*}
\sum_{j=2}^{n-1} x_1 y_j [e_1, e_j] + \sum_{i=2}^{n-1} x_i y_1 [e_i, e_1]
&= \sum_{j=2}^{n-1} x_1 y_j e_{j+1} + \sum_{i=2}^{n-1} x_i y_1 (-e_{i+1}) \\
&= \sum_{j=2}^{n-1} (x_1 y_j - x_j y_1)\, e_{j+1}.
\end{align*}
All other terms vanish, so we obtain the desired formula.
\end{proof}

\begin{Th}\label{Theoremm0n}
      A linear map $ \nabla : m_0(n) \to m_0(n)$  is an automorphism if and only if $ \nabla $ is of the following form:
\end{Th}
\begin{center}
\begin{enumerate}
    \item   $\nabla(e_1)=\sum_{i=1}^n a_i e_i,$\label{nabla(e_1)}
    \item   $ \nabla(e_2) =\sum_{i=1}^n b_ie_i,$\label{nabla(e_2)}
    \item   $\nabla(e_k) = a_1^{\,k-2}\sum_{j=2}^{\,n-k+2}b_j\,e_{j+k-2}. \qquad For \qquad 3\le k\le n, $\label{nabla(e_k)}
    \item    $\nabla(e_n)=a_1^{\,n-2} b_2 e_n.$\label{nabla(e_n)}

\end{enumerate}
\end{center}

\begin{proof}
For the basis vectors $e_1$ and $e_2$ we have
\[
\nabla(e_1)=\sum_{i=1}^n a_i e_i,\qquad
\nabla(e_2)=\sum_{i=2}^n b_i e_i,\qquad
\]
by Lemma~\eqref{lem:lower-triangular-unified}. In here, we obtain  \(a_1\neq0,\;b_2\neq0,\;\).

If \(a_1=0 \quad or \quad  b_2=0
\) This case fails to meet the criteria for an automorphism, as the determinant of its associated transformation matrix vanishes (i.e., equals zero), which consequently breaks bijectivity.

The automorphism condition together with the defining brackets property that
\begin{equation}\label{eq:recur-m0}
\nabla(e_{k+1}) = [\nabla(e_1),\nabla(e_k)] \qquad (2\le k\le n-1).
\end{equation}
With \(b_1=0\)  by Lemma~\eqref{lem:lower-triangular-unified} the third column becomes
\[
\nabla(e_3) =\sum_{j=2}^{n-1} (a_1 b_j - a_j b_1)\, e_{j+1}
= \sum_{j=2}^{n-1} (a_1 b_j - a_j \cdot 0)\, e_{j+1}
= a_1 \sum_{j=2}^{n-1} b_j\, e_{j+1}
\]
giving the diagonal entry \(a_{3,3}=a_1b_2\).  By induction we obtain the
simple closed form: for \(3\le k\le n\),
\begin{equation}\label{eq:closed-m0}
\nabla(e_k) = a_1^{\,k-2}\sum_{j=2}^{\,n-k+2}b_j\,e_{j+k-2},
\end{equation}
where we set \(b_j=0\) if \(j>n\) and the sum is zero when the upper bound
\(n-k+2<2\) (which happens only for \(k=n\)).

The formula holds for \(k=3\) as shown.  Assume it holds for some \(k\) (\(3\le k<n\)).
Then
\begin{align*}
\nabla(e_{k+1}) &= [\nabla(e_1),\nabla(e_k)] \\
  &= \Bigl[a_1 e_1+\sum_{i=2}^{n}a_i e_i,\;
           a_1^{k-2}\sum_{j=2}^{n-k+2}b_j e_{j+k-2}\Bigr] \\
  &= a_1^{k-2}\sum_{j=2}^{n-k+2}b_j\,
     \Bigl(a_1[e_1,e_{j+k-2}] + \sum_{i=2}^{n}a_i\,[e_i,e_{j+k-2}]\Bigr).
\end{align*}
Only the term \(a_1 e_1\) contributes to non-zero brackets, because
\([e_1,e_{m}]\) is the only non-zero bracket.  Hence
\[
\nabla(e_{k+1}) = a_1^{\,k-1}\sum_{j=2}^{n-k+2}b_j\,[e_1,e_{j+k-2}]
               = a_1^{\,k-1}\sum_{j=2}^{n-k+2}b_j\,e_{j+k-1}.
\]
The index \(j+k-1\le n\) requires \(j\le n-k+1\), so the sum becomes
\(\sum_{j=2}^{n-k+1}b_j\,e_{j+(k+1)-2}\), which is exactly the formula for \(k+1\).

Collecting the results, the matrix of an arbitrary automorphism of \(\mathfrak{m}_0(n)\)
in the basis \(\{e_1,\dots ,e_n\}\) is
\[
A_{m_0(n)} = \begin{pmatrix}\label{matrix m0n}
a_1 & 0      & 0              & 0                & \cdots & 0                  & 0 \\
a_2 & b_2 & 0              & 0                & \cdots & 0                  & 0 \\
a_3 & b_3 & a_1b_2 & 0                & \cdots & 0                  & 0 \\
a_4 & b_4 & a_1b_3 & a_1^2b_2 & \cdots & 0                  & 0 \\
\vdots   & \vdots  & \vdots         & \vdots           & \ddots & \vdots             & 0 \\
a_{n-1} & b_{n-1} & a_1b_{n-2} & a_1^2b_{n-3} & \cdots & a_1^{\,n-3}b_2 & 0 \\
a_n & b_n & a_1b_{n-1} & a_1^2b_{n-2} & \cdots & a_1^{\,n-3}b_3 & a_1^{\,n-2} b_2
\end{pmatrix},
\]

\end{proof}
\subsection{Automorphisms of the algebra \texorpdfstring{$m_2(n)$}{m2(n)}}\label{subsec3}

Now we construct automorphisms for the algebra $ m_2(n) $ by Table~\eqref{tab:filiformn}
All other brackets are zero .

For $x=\sum_{i=1}^n x_i e_i$ and $y=\sum_{j=1}^n y_j e_j$ in $m_2(n)$, the Lie bracket :

\begin{equation}\label{eq:brackets-m2n}
    [x,y] = \sum_{j=2}^{n-1} (x_1 y_j - x_j y_1)\, e_{j+1}
      + \sum_{j=3}^{n-2} (x_2 y_j - x_j y_2)\, e_{j+2}.
\end{equation}

The second sum is empty when $n=5$ (only $j=3$ gives a term) and non-empty for $n\ge6$.

For $k\ge 3$ we have $e_k=[e_1,e_{k-1}]$ (since $k-1\ge 2$).  Applying $\nabla$ gives

\begin{equation}\label{eq:brackets-m2nnablae_k}
    \nabla(e_k) = [\nabla(e_1),\nabla(e_{k-1})], \qquad k=3,\dots,n.
\end{equation}

Iterating $k-2$ times we obtain
\begin{equation}\label{eq:brackets-m2nnabla(e_k)}
    \nabla(e_k) = \underbrace{[\nabla(e_1),[\nabla(e_1),\dots,[\nabla(e_1),\nabla(e_2)]\dots]]}_{k-2\text{ times}}
          = ad(\nabla(e_1))^{k-2}(\nabla(e_2)).
\end{equation}

\begin{Th}\label{theorem m_2(n)}
    Let $\mathfrak{m}_2(n)$ with $n\ge5$ and the basis defined by Table~\eqref{tab:filiformn} .  A linear map $\nabla:{m}_2(n) \to {m}_2(n)$ is an automorphism if and only if there exist complex numbers
\[
a_i\ ,b_i, c_{i,j}\in C
\]
such that
\begin{align*}
\nabla(e_1) &= a_1 e_1 + \sum_{i=3}^{n} a_i e_i,\\[2mm]
\nabla(e_2) &= a_1^2 e_2 + \sum_{j=3}^{n} b_j e_j,\\[2mm]
\nabla(e_k)& = a_1^{k} e_k + a_1^{k-2} b_3 e_{k+1} + \sum_{j=4}^{n-k+2} a_1^{k-2} \Bigl( b_j - a_1 a_{j-1} \Bigr) e_{j+k-2}, \quad (3 \le k \le n-1),\\[2mm]
\nabla(e_n) &= a_1^n e_n.
 \end{align*}
\end{Th}

\begin{proof}

 we have by Lemma~\eqref{lem:lower-triangular-unified}
\[
b_1 = 0.
\]

Now we show $b_2=a_1^2$ and $a_2=0$

Using $\nabla(e_3)=[\nabla(e_1),\nabla(e_2)]$ we obtain
\[
\begin{aligned}
\nabla(e_3) = a_1 b_2 e_3 + a_1 b_3 e_4 + (-a_3 b_2 + a_2 b_3 + a_1 b_4) e_5 + (-a_4 b_2 + a_2 b_4 + a_1 b_5) e_6 + .. .
\end{aligned}
\]
Since $e_4=[e_1,e_3]$, we have $\nabla(e_4)=[\nabla(e_1),\nabla(e_3)]$.  A straightforward computation gives
\[
\begin{aligned}
\nabla(e_4) ={} & a_1^2 b_2 e_4 + \bigl(a_1 a_2 b_2 + a_1^2 b_3\bigr) e_5 + \bigl(-a_1 a_3 b_2 + 2a_1 a_2 b_3 + a_1^2 b_4\bigr) e_6 +... .
\end{aligned}
\]

There are two natural ways to compute $\nabla(e_5)$:
\[
\nabla(e_5) = [\nabla(e_1),\nabla(e_4)] \qquad\text{and}\qquad
\nabla(e_5) = [\nabla(e_2),\nabla(e_3)],
\]

{First method: $[\nabla(e_1),\nabla(e_4)]$}
\[
\begin{aligned}
\nabla(e_5) = a_1^3 b_2e_5+ \left(2a_1^2 a_2 b_2 + a_1^3 b_3\right)e_6+...
\end{aligned}
\]

Second method: $[\nabla(e_2),\nabla(e_3)]$
First compute $[\nabla(e_2),\nabla(e_3)]$:
\[
[\nabla(e_2),\nabla(e_3)] = \nabla(e_5)= a_1 b_2^2 e_5
+ a_1 b_2 b_3 e_6 +..
\]

Equating the two expressions for $\nabla(e_5)$
Comparing the coefficients of $e_5$ and $e_6$  we obtain:

\begin{itemize}
\item \textbf{Coefficient of $e_5$:}
  \[
  a_1^3b_2 = a_1b_2^2 \;\Longrightarrow\; b_2 = a_1^2.
  \]

\item \textbf{Coefficient of $e_6$:}
  \[
  a_1^3b_3 + 2a_1^2a_2b_2(2+\alpha) = a_1b_2b_3.
  \]
  Substituting $b_2 = a_1^2$ and dividing by $a_1^3$ (since $a_1\neq0$) yields

  Actually, $a_1b_2b_3 = a_1\cdot a_1^2 b_3 = a_1^3 b_3$, so the right-hand side is $a_1^3b_3$.  The left-hand side is $a_1^3b_3 + 2a_1^2a_2b_2(2+\alpha) = a_1^3b_3 + 2a_1^2a_2\cdot a_1^2(2+\alpha) = a_1^3b_3 + 2a_1^4a_2(2+\alpha)$.  Thus we get
  \[
  a_1^3b_3 + 2a_1^4a_2(2+\alpha) = a_1^3b_3 \;\Longrightarrow\; 2a_1^4a_2(2+\alpha)=0.
  \]
  Since $\alpha\neq-2$ (generic case) and $a_1\neq0$, we deduce $a_2=0$.

\end{itemize}

so we have
\[
\nabla(e_k) = [\nabla(e_1),\nabla(e_{k-1})] = \underbrace{[\nabla(e_1),[\nabla(e_1),\dots,[\nabla(e_1),\nabla(e_2)]\dots]]}_{k-2\text{ times}}
          = ad(\nabla(e_1))^{k-2}(\nabla(e_2)).
\]
Then, for every $3\le k\le n$,
\[
\nabla(e_k)=a_1^{k-2} b_2 e_k + a_1^{k-2} b_3 e_{k+1} + \sum_{m=k+2}^{n} a_1^{k-2} \bigl( b_{m-k+2} - a_1 a_{m-k+1} \bigr) e_m,
\tag{1}
\]
with the convention that the sum is zero if $k+2>n$, and $b_m=0$ for $m>n$.

We proceed by induction on $k$.

\medskip
\noindent\textbf{Base case: $k=3$.}
Since $e_3=[e_1,e_2]$, we have
\[
\nabla(e_3)=[\nabla(e_1),\nabla(e_2)].
\]
Using $\nabla(e_1)=a_1 e_1+\sum_{i=3}^n a_i e_i$ (because $a_2=0$) and $\nabla(e_2)=b_2 e_2+\sum_{j=3}^n b_j e_j$, and the fact that $[e_i,e_j]=0$ for all $i,j\ge3$, we obtain
\[
\nabla(e_3)=a_1 b_2 e_3 + a_1 \sum_{j=3}^{n-1} b_j e_{j+1} - b_2 \sum_{i=3}^{n-2} a_i e_{i+2}.
\]
Changing indices in the sums and using $b_2=a_1^2$, we get
\[
\nabla(e_3)=a_1 b_2 e_3 + a_1 b_3 e_4 + \sum_{m=5}^{n} a_1\bigl( b_{m-1} - a_1 a_{m-2} \bigr) e_m.
\]
This is exactly formula (1) for $k=3$ (since $a_1^{3-2}=a_1$).  Thus the base case holds.

\medskip
\noindent\textbf{Inductive step.}
Assume that formula (1) holds for some $k$ with $3\le k\le n-1$.  We prove it for $k+1$.

Since $e_{k+1}=[e_1,e_k]$, we have
\[
\nabla(e_{k+1})=[\nabla(e_1),\nabla(e_k)].
\]
As before, $\nabla(e_1)=a_1 e_1+\sum_{i=3}^n a_i e_i$.  Because $\nabla(e_k)$ is a linear combination of $e_k,e_{k+1},\dots,e_n$ (by the induction hypothesis), and $[e_i,e_j]=0$ for all $i,j\ge3$, the terms $\sum_{i=3}^n a_i e_i$ give no contribution.  Hence
\[
\nabla(e_{k+1})=a_1 [e_1,\nabla(e_k)].
\]

Substitute the expression for $\nabla(e_k)$ from the induction hypothesis:
\[
\begin{aligned}
\nabla(e_{k+1})
&= a_1 [e_1,\; a_1^{k-2} b_2 e_k + a_1^{k-2} b_3 e_{k+1} + \sum_{m=k+2}^{n} a_1^{k-2} \bigl( b_{m-k+2} - a_1 a_{m-k+1} \bigr) e_m ] \\
&= a_1^{k-1} b_2 [e_1,e_k] + a_1^{k-1} b_3 [e_1,e_{k+1}] + \sum_{m=k+2}^{n} a_1^{k-1} \bigl( b_{m-k+2} - a_1 a_{m-k+1} \bigr) [e_1,e_m].
\end{aligned}
\]
Using $[e_1,e_s]=e_{s+1}$ for $s\le n-1$, and $[e_1,e_n]=0$, we get
\[
\begin{aligned}
\nabla(e_{k+1})
&= a_1^{k-1} b_2 e_{k+1} + a_1^{k-1} b_3 e_{k+2} \\
&\quad + \sum_{m=k+2}^{n-1} a_1^{k-1} \bigl( b_{m-k+2} - a_1 a_{m-k+1} \bigr) e_{m+1}.
\end{aligned}
\]
Now we re-index the sum: let $j=m+1$.  Then when $m=k+2$, $j=k+3$, and when $m=n-1$, $j=n$.  Thus
\[
\nabla(e_{k+1})
= a_1^{k-1} b_2 e_{k+1} + a_1^{k-1} b_3 e_{k+2}
+ \sum_{j=k+3}^{n} a_1^{k-1} \bigl( b_{j-(k+1)+2} - a_1 a_{j-(k+1)+1} \bigr) e_j.
\]
This is precisely formula (1) with $k$ replaced by $k+1$, because $(k+1)-2 = k-1$.

\medskip
\noindent Thus, by induction, formula (1) holds for all $3\le k\le n$.  For $k=n$, the sum is empty and we obtain
\[
\nabla(e_n)=a_1^{n-2} b_2 e_n = a_1^{n} e_n
\]

\textbf{Matrix form for $n\ge 6$}

Collecting the results, the matrix of \texorpdfstring{$\nabla$}{nabla} in the basis $\{e_1,\dots,e_n\}$ is lower triangular:

\[
A = \begin{pmatrix}\label{matrix m2n}
a_1      & 0        & 0        & 0        & \cdots & 0        & 0 \\
0      & a_1^2    & 0        & 0        & \cdots & 0        & 0 \\
a_3      & b_3      & a_1^3    & 0        & \cdots & 0        & 0 \\
a_4      & b_4      & a_1 b_3  & a_1^4    & \cdots & 0        & 0 \\
\vdots   & \vdots   & \vdots   & \vdots   & \ddots & \vdots   & \vdots \\
a_{n-1}  & b_{n-1}  & c_{n-1,3} & c_{n-1,4} & \cdots & a_1^{\,n-1} & 0 \\
a_n      & b_n      & c_{n,3}   & c_{n,4}   & \cdots & c_{n,n-1}  & a_1^{\,n}
\end{pmatrix},
\]
where the entries $c_{j,k}$ for $k\ge3$ are defined recursively
\begin{equation}\label{eq:coefficients-Cjk}
C(j, k) = \begin{cases}
a_1^{k-2} (b_{j-1} - a_1 a_{j-2}), & \text{if } 5 \le j \le n-k+3.
\end{cases}
\end{equation}
\end{proof}

We can similarly prove the following theorems.

\subsection{Automorphisms of the algebra \texorpdfstring{$W^+(n)$}{W+(n)}}\label{subsec4}
Let $n \ge 5$ and consider the $n$-dimensional complex Lie algebra $W^+(n)$
with basis $\{e_1,\dots,e_n\}$ and non-zero brackets introduce the following Table~\eqref{tab:filiformn}.
For $x=\sum_{i=1}^n x_i e_i$ and $y=\sum_{j=1}^n y_j e_j$ in $W^+(n)$
\begin{equation}\label{eq:bracket}
[x,y] = \sum_{1\le i<j,\; i+j\le n} (j-i)(x_i y_j - x_j y_i)\, e_{i+j}.
\end{equation}

\begin{Th}\label{W+(n)umumiy}
Let $W^+(n)$ with $n\ge5$ and the basis defined in Table 1.  A linear map $\nabla:W^+(n) \to W^+(n)$ is an automorphism if and only if there exist complex numbers
\[
a_i, b_i\in C,\quad
\]
such that
\[
\begin{aligned}
\nabla(e_1) &= a_1e_1+\sum_{i=3}^{n} a_i e_i,\\[2mm]
\nabla(e_2) &= a_1^2 e_2 + \sum_{j=3}^{n} b_j e_j,\\[2mm]
\nabla(e_k) &=\sum_{p=k}^{n} C(p,k)\, e_p,\qquad 1\le k\le n., \qquad where \\[2mm]
C(p,k)&=\frac{1}{k-2}\sum_{i=1}^{\lfloor (p-1)/2 \rfloor}
      (p-2i)\Bigl( a_i\, C(p-i,k-1) - a_{p-i}\, C(i,k-1) \Bigr); \quad C(p,1)=a_p,\quad C(p,2)=b_p.\\[2mm]
\nabla(e_n) &= a_1^n e_n.
\end{aligned}
\]
\end{Th}

\subsection{Automorphisms of the algebra \texorpdfstring{$m_{0,1}(2k+1)$}{m(2k+1)}}\label{subsec5}
First, we consider automorphisms of the algebra \(m_{0,1}(2k+1)\). For the algebra \(m_{0,1}(2k+1)\), we introduce the following Table~\eqref{tab:filiformn}.

For $x=\sum_{i=1}^n x_i e_i$, $y=\sum_{j=1}^n y_j e_j$,
the Lie bracket derived from Table~ \eqref{tab:filiformn} is
\begin{equation}\label{eq:bracketm(n)}
[x,y] = \sum_{j=2}^{n-1} (x_1y_j - x_jy_1) e_{j+1}
      + \sum_{l=2}^{(n-1)/2} (-1)^{l+1} (x_l y_{n-l} - x_{n-l} y_l) e_n .
\end{equation}
\text{This algebra differs from \(\mathfrak{m}_0(n)\) only by the coefficient of the last basis vector \(e_n\).}

$$[x,y] = [x,y]_{m_0(n)}
      + \sum_{l=2}^{(n-1)/2} (-1)^{l+1} (x_l y_{n-l} - x_{n-l} y_l) e_n .$$

For the algebra \(m_0(n)\), the expression for \(\nabla(e_t)\) was of the following form $\nabla(e_t) = a_1^{\,t-2}\sum_{j=2}^{\,n-t+2}b_j\,e_{j+t-2}$. Accordingly, for \(m_{0,1}(n)\) it takes the following form.

\begin{equation}\label{m01nnablaet}
    \nabla(e_t) = a_1^{\,t-2}\sum_{j=2}^{\,n-t+2}b_j\,e_{j+t-2}+ a_1^{\,t-3}\sum_{l=2}^{(n-1)/2} (-1)^{l+1} (a_l b_{n-t-l+3} - a_{n-l} b_{l-t+3}) e_n
\end{equation}
in here $a_m,b_m=0$ if $m\le 0 \quad or \quad m>n$

\begin{Th}\label{thm:mainm(2k+1)}
Let ${m}_{0,1}(n)$ with $k\ge3$.  A linear map $\nabla:{m}_{0,1}(n) \to {m}_{0,1}(n)$
is an automorphism if and only if there exist complex numbers
\[
a_i\in C,\quad  b_i\in C,
\]
such that
\[
\begin{aligned}
\nabla(e_1)&=\sum_{i=1}^{n} a_i e_i, \\[2mm]
\nabla(e_2)&=a_1^{2}e_2+\sum_{j=3}^{n} b_j e_j,\\[2mm]
\nabla(e_t) &= a_1^{\,t-2}\sum_{j=2}^{\,n-t+2}b_j\,e_{j+t-2}+ a_1^{\,t-3}\sum_{l=2}^{(n-1)/2} (-1)^{l+1} (a_l b_{n-t-l+3} - a_{n-l} b_{l-t+3}) e_n ,\\[2mm]
\nabla(e_n)&= a_1^n e_n, \\[2mm]
\end{aligned}
\]
\end{Th}

\subsection{Automorphisms of the algebra \texorpdfstring{$m_{0,2}(2k+2)$}{m(2k+2)}}\label{subsec6}
we consider automorphisms of the algebra \(m_{0,2}(2k+2)\). For the algebra \(m_{0,2}(n)\) (in here $n=2k+2$), we introduce the following relations by Table~\eqref{tab:filiformn}.

For \(x=\sum_{i=1}^n x_i e_i\) and \(y=\sum_{j=1}^n y_j e_j\), the Lie bracket
derived from  is \(m_{0,2}(n)\)
\begin{equation}\label{eq:bracketm(2k+2)}
\begin{aligned}
\relax[x,y] &= \sum_{j=2}^{n-1} (x_1y_j - x_jy_1)\, e_{j+1} \\
      &+ \sum_{l=2}^{\frac{n}{2}-1} (-1)^{l+1} (x_l y_{n-1-l} - x_{n-1-l} y_l)\, e_{n-1} \\
      &+ \sum_{j=2}^{\frac{n}{2}-1} (-1)^{j+1} (\frac{n}{2}-j) (x_j y_{n-j} - x_{n-j} y_j)\, e_n .
\end{aligned}
\end{equation}

In the general situation, the algebra under consideration is obtained from \(m_{0,1}(n)\) by the substitution \(n' = n+1\), i.e. it is precisely the algebra \(m_{0,2}(n')\). Consequently, the automorphism matrix coincides with that of \(m_{0,1}(n)\) in the first \(n-1\) rows, and the only difference occurs in the last row, which is governed by the bracket
\[
[e_p, e_{n-p}] = (-1)^{p+1}\left(\frac{n}{2}-p\right)e_{n}, \qquad 2\le p \le \frac{n}{2}-1.
\]

so we compute $\nabla(e_t)=[\nabla(e_1),\nabla(e_{t-1})]$

\begin{equation}\label{m02nnablaet}
\begin{aligned}
    \nabla(e_t) &= a_1^{\,t-2}\sum_{j=2}^{\,n-t+2}b_j\,e_{j+t-2}+ a_1^{\,t-3}\sum_{l=2}^{(n-1)/2} (-1)^{l+1} (a_l b_{n-t-l+2} - a_{n-l-1} b_{l-t+3}) e_{n-1}+ \\ &  a_1^{\,t-3}\sum_{p=2}^{(n-1)/2} (-1)^{p+1}(\frac{n}{2}-p)  (a_p b_{n-p-t+3} - a_{n-p} b_{p-t+3}) e_{n}
\end{aligned}
\end{equation}
in here $a_m,b_m=0$ if $m\le 0$ or $m>n$

\begin{Th}\label{thm:autm(2k+2)}
Let $m_{0,2}(n)$ be the Lie algebra.
A linear map
\[
\nabla:m_{0,2}(n)\longrightarrow m_{0,2}(n)
\]
is an automorphism if and only if there exist complex numbers
\[
a_1\neq0,\qquad
a_i,b_i,\in C
\]
such that
\[
\begin{aligned}
\nabla(e_1)&=a_1e_1+\sum_{i=3}^{n}a_i e_i,\\[2mm]
\nabla(e_2)&=a_1^2e_2+\sum_{j=3}^{n}b_je_j,\\[2mm]
\nabla(e_t) &= a_1^{\,t-2}\sum_{j=2}^{\,n-t+2}b_j\,e_{j+t-2}+ a_1^{\,t-3}\sum_{l=2}^{(n-1)/2} (-1)^{l+1} (a_l b_{n-t-l+2} - a_{n-l-1} b_{l-t+3}) e_{n-1}+ \\ &  a_1^{\,t-3}\sum_{p=2}^{(n-1)/2} (-1)^{p+1}(\frac{n}{2}-p)  (a_p b_{n-p-t+3} - a_{n-p} b_{p-t+3}) e_{n} \\[2mm]
\nabla(e_{n})&=a_1^{\,n}e_{n},\\[2mm]
\end{aligned}
\]
\end{Th}

\subsection{Automorphisms of the algebra \texorpdfstring{$m_{0,3}(2k+3)$}{m(2k+3)}}\label{subsec7}
we consider automorphisms of the algebra \(m_{0,3}(2k+3)\). For the algebra \(m_{0,3}(2k+3)\), we introduce the following relations by Table~\eqref{tab:filiformn}.

In the general situation, the algebra under consideration is obtained from \(m_{0,2}(n)\) by the substitution \(n' = n+1\), i.e. it is precisely the algebra \(m_{0,3}(n')\). Consequently, the automorphism matrix coincides with that of \(m_{0,2}(n)\) in the first \(n-1\) rows, and the only difference occurs in the last row, which is governed by the bracket
$$[e_m, e_{2k-m+3}] = (-1)^m\left((m-2)k - \frac{(m-2)(m-1)}{2}\right)e_{2k+3},\ m = 3, \ldots, k+1 \label{eq:defm(2k+3)}$$

For $x=\sum_{i=1}^n x_i e_i$, $y=\sum_{j=1}^n y_j e_j$ in $m_{0,3}(n)$ , the Lie bracket
derived from \eqref{eq:defm(2k+3)} is
\begin{equation}\label{eq:bracketm(2k+3)}
\begin{aligned}
\relax [x,y] &= [x,y]_{m_{0,2}(n')} + \sum_{m=3}^{k+1} (-1)^m \Bigl((m-2)k - \frac{(m-2)(m-1)}{2}\Bigr)
         (x_m y_{2k+3-m} - x_{2k+3-m} y_m)\, e_{2k+3} .
\end{aligned}
\end{equation}

\begin{Th}\label{m03nmain}
    Let $m_{0,3}(n)$ be the Lie algebra
A linear map
\[
\nabla:m_{0,3}(n)\longrightarrow m_{0,3}(n)
\]
is an automorphism if and only if there exist complex numbers

\[
a_1\neq0,\qquad
a_i,b_i,\in C
\]

such that

\[
\begin{aligned}
\nabla(e_1)&=a_1e_1+\sum_{i=3}^{n}a_i e_i,\\[2mm]
\nabla(e_2)&=a_1^2e_2+\sum_{j=3}^{n}b_je_j,\\[2mm]
\nabla(e_t) &= a_1^{\,t-2}\sum_{j=2}^{\,n-t+2}b_j\,e_{j+t-2}+ a_1^{\,t-3}\sum_{l=2}^{(n-3)/2} (-1)^{l+1} (a_l b_{n-t-l+1} - a_{n-l-2} b_{l-t+3}) e_{n-2}+ \\ &  a_1^{\,t-3}\sum_{p=2}^{(n-3)/2} (-1)^{p+1}(\frac{n-1}{2}-p)  (a_p b_{n-p-t+2} - a_{n-p} b_{p-t+2}) e_{n-1}+ \\ &  a_1^{\,t-3}\sum_{m=2}^{(n-3)/2} (-1)^{m+1}(m  \frac{n-3}{2}-\frac{(m-2)(m-1)}{2})  (a_m b_{n-m -t+3} - a_{n-m} b_{m-t+3}) e_{n} \\[2mm]
\nabla(e_{n})&=a_1^{\,n}e_{n},\\[2mm]
\end{aligned}
\]
in here $a_k,b_k=0$ if $k\le 0 \quad or \quad k>n$.
\end{Th}

\bigskip

\section{Local automorphisms }\label{sec4}

\begin{defn}\label{defnlocal}
A linear map $\nabla:g\to g$ is a \emph{local automorphism} if for every $x\in g$ there exists an automorphism $\Phi_x$ such that $\nabla(x)=\Phi_x(x)$.
\end{defn}

Let $B=(b_{i,j})$ be the matrix of $\nabla$ with respect to $\{e_1,\dots,e_n\}$. For $x=\sum_{i=1}^n x_i e_i$, write $\bar{x}=(x_1,\dots,x_n)^T$. If $\Phi_x$ has matrix $A_x$, then the local condition is equivalent to the vector equation
\begin{equation}\label{systemlocal}
    B\bar{x}=A_x\bar{x}.
\end{equation}

\begin{Lem}\label{lem:bijective}
Every local automorphism \( \nabla \) of \( \mathfrak{g} \) is bijective.
\end{Lem}

\begin{proof}
If \( \nabla(x)=0 \), then for an automorphism \( \Phi_x \) with \( \nabla(x)=\Phi_x(x) \), we have \( \Phi_x(x)=0 \), hence \( x=0 \). Thus \( \nabla \) is injective. Since \( \mathfrak{g} \) is finite-dimensional, \( \nabla \) is bijective.
\end{proof}

\begin{Lem}\label{lem:lcsinv}
Every local automorphism \( \nabla \) preserves the lower central series: \( \nabla(\mathfrak{g}^r) = \mathfrak{g}^r \) for all \( r \).
\end{Lem}

\begin{proof}
Let \( y \in \mathfrak{g}^r \). Then \( y = \sum_j [u_j, v_j] \) with \( u_j \in \mathfrak{g}^{r-1} \). For the local automorphism \( \nabla \), choose \( \Phi_y \) such that \( \nabla(y) = \Phi_y(y) \). Since \( \Phi_y \) is an automorphism, \( \Phi_y(\mathfrak{g}^r) = \mathfrak{g}^r \). Hence \( \nabla(y) = \Phi_y(y) \in \mathfrak{g}^r \). Thus \( \nabla(\mathfrak{g}^r) \subseteq \mathfrak{g}^r \). By Lemma~\eqref{lem:bijective}, \( \dim \nabla(\mathfrak{g}^r) = \dim \mathfrak{g}^r \), so \( \nabla(\mathfrak{g}^r) = \mathfrak{g}^r \).
\end{proof}

\subsection{Local automorphisms of the algebra \texorpdfstring{$g_{k,\alpha}$}{g(kn)}}\label{subsec8}

\begin{Th}\label{thm:local-gk}
A linear map $\nabla$ on the one-parameter family of Lie algebras $\mathfrak{g}_{k,\alpha}$ ($k \in \{7,8,9,10,11\}$) is a local automorphism if and only if its matrix representation $M_{\nabla} = (b_{m,j})$ with respect to the homogeneous basis $\{e_1, e_2, \dots, e_k\}$ is lower triangular.
\begin{equation}\label{eq:local-matrix-gk-form}
M_{\nabla} = \begin{pmatrix}
b_{1,1} & 0 & 0 & 0 & \cdots & 0 & 0 \\
0 & b_{2,2} & 0 & 0 & \cdots & 0 & 0 \\
b_{3,1} & b_{3,2} & b_{3,3} & 0 & \cdots & 0 & 0 \\
b_{4,1} & b_{4,2} & b_{4,3} & b_{4,4} & \cdots & 0 & 0 \\
\vdots & \vdots & \vdots & \vdots & \ddots & \vdots & 0 \\
b_{n-1,1} & b_{n-1,2} & b_{n-1,3} & b_{n-1,4} & \cdots & b_{n-1,n-1} & 0 \\
b_{n,1} & b_{n,2} & b_{n,3} & b_{n,4} & \cdots & b_{n,n-1} & b_{n,n}
\end{pmatrix},
\end{equation}
where $b_{m,j} \in \mathbb{C}$ and $d_{1,1} \neq 0$.
\end{Th}

\begin{proof}
Let $\nabla $ be an arbitrary local automorphism on $\mathfrak{m}_0(n)$, and, let $ M_{\nabla} $ be the matrix of $ \nabla$. By the definition for
all $x = x_1e_1 + \cdot\cdot\cdot + x_ne_n \in L_n$ there exists an automorphism $\phi_x $ on $ \mathfrak{m}_0(n) $ with respect to the basis 
$\{e_1, e_2, . . . , e_n\} $ of $m_0(n)$
$$\nabla(x)=\phi_x(x).$$

By Theorem~\eqref{Theoremm0n}  the automorphism $\phi_x$ has the following matrix form:
\begin{equation}
A_x = \begin{pmatrix}\label{matrix_AX_gk}
a_1^x      & 0        & 0        & 0        & 0 & 0        & \dots & 0          & 0 \\
0        & (a_1^x)^2    & 0        & 0        & 0 & 0        & \dots & 0          & 0 \\
a_3^x      & b_3^x      & (a_1^x)^3    & 0        & 0 & 0        & \dots & 0          & 0 \\
a_4^x      & b_4^x      & a_1^x b_3^x  & (a_1^x)^4      & 0 & 0        & \dots & 0         & 0 \\
a_5^x      & b_5^x      & c_{5,3}^x  & (a_1^x)^2 b_3^x  & (a_1^x)^5 & 0    & \dots & 0         & 0 \\
a_6^x      & b_6^x      & c_{6,3}^x  & c_{6,4}^x  & (a_1^x)^3 b_3^x   & (a_1^x)^6     & \dots         & 0  & 0 \\
\vdots   & \vdots   & \vdots   & \vdots   & \vdots & \vdots    & \ddots & \vdots & \vdots\\
a_{n-1}^x  & b_{n-1}^x  & c_{n-1,3}^x& c_{n-1,4}^x & c_{n-1,5}^x & c_{n-1,6}^x & \dots & (a_1^x)^{n-1}  & 0 \\
a_{n}^x    & b_{n}^x    & c_{n,3}^x  & c_{n,4}^x   & c_{n,5}^x   & c_{n,6}^x   & \dots & (a_1^x)^{n-3}b_3^x & (a_1^x)^n
\end{pmatrix}.
\end{equation}
where  $c_{i,j}^x $ is determined in accordance with Theorem~\eqref{theorem gk aut}

Sequentially substituting the basis elements \(x = e_1, \dots, e_n\) into the local automorphism condition \(\nabla(x) = \phi_x(x)\) (or equivalently \(M_\nabla\bar{x} = A_x\bar{x}\), where \(\bar{x}\)) represents the coordinate vector of \({x}\)) implies that the matrix \(A = (b_{i,j})\) is lower triangular (\(b_{i,j} = 0\) for \(i < j\)) with non-zero diagonal entries (\(b_{k,k} \neq 0\) for \(1 \le k \le n\)). This structural property relies fundamentally on the underlying field of complex numbers \(\mathbb{C}\), over which the algebra $g_{k, \alpha}$ is defined. Furthermore, applying the relation \(A\bar{x} = A_x\bar{x}\) to an arbitrary element \(x = \sum_{i=1}^n x_i e_i\) yields the following system of algebraic equations:
\begin{equation}\label{local system gk}
    \begin{cases}
        b_{1,1}x_1=a_1^x x_1, \\
        b_{2,2}x_2=(a_1^x)^2 x_2, \\
        b_{3,1}x_1+b_{3,2}x_2+b_{3,3}x_3=a_3^x x_1+b_3^x x_2+(a_1^x)^3 x_3, \\
        \sum_{j=1}^i b_{i,j}x_j = a_i^x x_1 + b_i^x x_2 + \sum_{j=3}^{i-2} c_{i,j}^x x_j+(a_1^x)^{i-3}b_3^x x_{i-1}+(a_1^x)^i x_i, \quad 1 \le i \le k,
    \end{cases}
\end{equation}

in terms of the variables \(a_{i}^{x}\) $(1 \leq i \leq n)$ and  \(b_{i}^{x}\) $(2 \leq i \leq n)$.

We shall now demonstrate that the linear operator represented by the matrix $A$
 a local automorphism. By definition, if for every element $x \in g_{k, \alpha}$
there exists a matrix $A_x$ of the form given in Theorem~\eqref{theorem gk aut} satisfying the identity
$$A_x \bar{x} = A \bar{x}$$
then the corresponding linear operator is a local automorphism.
Equivalently, the linear operator defined by $A$ is a local automorphism if, for each
$x \in g_{k, \alpha}$, the system of linear equations~\eqref{local system gk} possesses a valid solution with
respect to the variables

$a_1^x, a_3^x, \dots, a_n^x, b_2^x, \dots, b_n^x $.

To establish this, we analyze
the following cases:

\begin{itemize}
    \item  If $x_1 \neq 0$, then we directly obtain $a_1^x = b_{1,1}$. For each $i$ satisfying $2 \le i \le n$, the remaining unknowns $a_i^x$ can be  determined by the following explicit formula:
$$a_i^x = b_{i,1} + \frac{1}{x_1}\Bigl( (b_{i,2}-b_i^x)x_2+\sum_{j=3}^i (b_{i,j} - \bigl(a_1^x\bigl)^{j-2} b_{i-j+2}^x) x_j\Bigl),$$
where $2 \le i \le n$, and these equalities give a solution of the system of equations~\eqref{local system gk}. Here  $b_{2}^x,b_{3}^x, \cdots , b_{n}^x$ are defined arbitrary.[Indeed, since the system of linear equations~\eqref{local system gk} generally comprises \(n\) equations and \(2n-1\) unknowns, it exhibits \(n-1\) degrees of freedom. Consequently, these particular parameters can be chosen arbitrarily with respect to the remaining variables to guarantee the consistency of the system.]
\item  if $x_1=0$ and $x_2 \neq 0$, then $b_{2,2}=b_2^x$ for each $i$ satisfying $3 \le i \le n$, the remaining unknowns $a_i^x$ can be recursively determined by the following explicit formula:

$a_3^x=\frac{1}{x_1}\bigl( \sum_{j=1}^{3}b_{3,j}x_j -  b_3^x x_2 -(b_{1,1})^3 x_3  \bigl)      $

$\cdots \cdots \cdots \cdots$

$a_i^x=\frac{1}{x_1} \bigl(   \sum_{j=1}^{i}b_{i,j}x_j-\sum_{j=3}^{i-2}c_{i,j}^x x_j-(b_{1,1})^{i-3}b_3^x x_{i-1}-(b_{1,1})^ix_i \bigl)$

 Here  $b_i^x$ $3\le i \le n$ are defined arbitrary.
\item  if $x_1=0$ and $x_2 \neq 0$, then $b_{2,2}=(a_1^x)^2 \Longrightarrow a_1^x=\sqrt{b_{2,2}}$ for each $i$ satisfying $3 \le i \le n$, the remaining unknowns $b_i^x$ can be recursively determined by the following explicit formula:

$b_3^x=\frac{1}{x_2}\bigl( \sum_{j=2}^{3}b_{3,j}x_j -  (\sqrt{b_{2,2}})^3 x_3  \bigl)      $

$\cdots \cdots \cdots \cdots$

$b_i^x=\frac{1}{x_2} \bigl(   \sum_{j=2}^{i}b_{i,j}x_j-\sum_{j=3}^{i-2}c_{i,j}^x x_j-(\sqrt{b_{2,2}})^{i-3}b_3^x x_{i-1}-(\sqrt{b_{2,2}})^ix_i \bigl)$

give solution of the system of equation~\eqref{local system gk}. Here the variables $a_i^x$ $3\le i \le n$ is arbitrary value.

\item If $ x_1=x_2=\cdots= x_{m-1} = 0$ and $x_m \neq 0, $ then the following equalities, $m \le i \le n$

$b_{m,m}={(a_1^x)^{m}} \Longrightarrow a_1^x=\sqrt[m]{b_{m,m}}$

$b^x_{i-m+2}= \frac{1}{a_1^{\,m-2} x_m} \bigl(\sum_{j=m}^{i}b_{i,j} -Tx_m- \sum_{j=m+1}^{i-2} c_{i,j}^x x_j \\ -(\sqrt[m]{b_{m,m}})^{i-3}b_3^x x_{i-1}-(\sqrt[m]{b_{m,m}})^i x_i\bigl); \quad m \le i \le k, \quad 3 \le m \le k$

give solution of the system of equation~\eqref{local system gk}. Here the variables $a_i^x$ is arbitrary value.

\end{itemize}
Consequently, the system of linear equations~\eqref{local system gk} consistently possesses a solution
for every element $x \in g_{k, \alpha}$, which implies that the linear operator induced by
the matrix $A$ constitutes a local automorphism. Owing to the fact that the local
automorphism $\nabla$ was chosen arbitrarily and is pointwise defined by the
automorphisms as characterized in Theorem~\eqref{theorem gk aut}, it follows that the matrix
representation of any local automorphism of the algebra $g_{k, \alpha}$ necessarily admits
a lower triangular form.
\end{proof}

\subsection{Local automorphisms of the algebra \texorpdfstring{$m_0(n)$}{m(0n)}}\label{subsec9}

\begin{Th}\label{thm:local-m0}
A linear map $\nabla$ on the filiform Lie algebra $\mathfrak{m}_0(n)$ ($n \ge 3$) is a local automorphism if and only if its matrix representation $M_{\nabla} = (d_{i,j})_{n \times n}$ with respect to the homogeneous basis $\{e_1, e_2, \dots, e_n\}$ is lower triangular and invertible, i.e.,
\begin{equation}\label{eq:local-matrix-form}
M_{\nabla} = \begin{pmatrix}
b_{1,1} & 0 & 0 & 0 & \cdots & 0 & 0 \\
b_{2,1} & b_{2,2} & 0 & 0 & \cdots & 0 & 0 \\
b_{3,1} & b_{3,2} & b_{3,3} & 0 & \cdots & 0 & 0 \\
b_{4,1} & b_{4,2} & b_{4,3} & b_{4,4} & \cdots & 0 & 0 \\
\vdots & \vdots & \vdots & \vdots & \ddots & \vdots & 0 \\
b_{n-1,1} & b_{n-1,2} & b_{n-1,3} & b_{n-1,4} & \cdots & b_{n-1,n-1} & 0 \\
b_{n,1} & b_{n,2} & b_{n,3} & b_{n,4} & \cdots & b_{n,n-1} & b_{n,n}
\end{pmatrix},
\end{equation}
where $b_{i,j} \in \mathbb{C}$ and $b_{i,i} \neq 0$ for all $1 \le i \le n$.
\end{Th}

\begin{proof}
Let $\nabla $ be an arbitrary local automorphism on $\mathfrak{m}_0(n)$, and, let $ M_{\nabla} $ be the matrix of $ \nabla$. By the definition for
all $x = x_1e_1 + \cdot\cdot\cdot + x_ne_n \in L_n$ there exists an automorphism $\phi_x $ on $ \mathfrak{m}_0(n) $ with respect to the basis $ \{e_1, e_2, . . . , e_n\} $ of $m_0(n)$
$$\nabla(x)=\phi_x(x).$$

By Theorem~\eqref{Theoremm0n}  the automorphism $\phi_x$ has the following matrix form:
\[
A_x = \begin{pmatrix}\label{A_x form matrix m0n}
a_1^x & 0 & 0 & 0 & \cdots & 0 & 0 \\
a_2^x & b_2^x & 0 & 0 & \cdots & 0 & 0 \\
a_3^x & b_3^x & a_1^x b_2^x & 0 & \cdots & 0 & 0 \\
a_4^x & b_4^x & a_1^x b_3^x & (a_1^x)^2 b_2^x & \cdots & 0 & 0 \\
\vdots & \vdots & \vdots & \vdots & \ddots & \vdots & 0 \\
a_{n-1}^x & b_{n-1}^x & a_1^x b_{n-2}^x & (a_1^x)^2 b_{n-3}^x & \cdots & (a_1^x)^{n-3} b_2^x & 0 \\
a_n^x & b_n^x & a_1^x b_{n-1}^x & (a_1^x)^2 b_{n-2}^x & \cdots & (a_1^x)^{n-3} b_3^x & (a_1^x)^{n-2} b_2^x
\end{pmatrix}.
\]
Sequentially substituting the basis elements \(x = e_1, \dots, e_n\) into the local automorphism condition \(\nabla(x) = \phi_x(x)\) (or equivalently \(M_\nabla\bar{x} = A_x\bar{x}\), where \(\bar{x}\)) represents the coordinate vector of \({x}\)) implies that the matrix \(A = (b_{i,j})\) is lower triangular (\(b_{i,j} = 0\) for \(i < j\)) with non-zero diagonal entries (\(b_{k,k} \neq 0\) for \(1 \le k \le n\)). This structural property relies fundamentally on the underlying field of complex numbers \(\mathbb{C}\), over which the algebra $\mathfrak{m}_0(n)$ is defined. Furthermore, applying the relation \(A\bar{x} = A_x\bar{x}\) to an arbitrary element \(x = \sum_{i=1}^n x_i e_i\) yields the following system of algebraic equations:

\begin{equation}\label{local sysytem m0n}
    \sum_{j=1}^{i}b_{i,j}x_j=a_i^x x_1+b_i^x x_2+\sum_{j=3}^{i}\bigl(a_1^x\bigl)^{j-2} b_{i-j+2}^xx_j
\end{equation}

in terms of the variables \(a_{i}^{x}\) $(1 \leq i \leq n)$ and  \(b_{i}^{x}\) $(2 \leq i \leq n)$.

We shall now demonstrate that the linear operator represented by the matrix $A$
 a local automorphism. By definition, if for every element $x \in \mathfrak{m}_0(n)$
there exists a matrix $A_x$ of the form given in Theorem~\eqref{Theoremm0n} satisfying the identity
$$A_x \bar{x} = A \bar{x}$$
then the corresponding linear operator is a local automorphism.
Equivalently, the linear operator defined by $A$ is a local automorphism if, for each
$x \in \mathfrak{m}_0(n)$, the system of linear equations~\eqref{local sysytem m0n} possesses a valid solution with
respect to the variables

$a_1^x, a_2^x, \dots, a_n^x, b_2^x, \dots, b_n^x $.

To establish this, we analyze
the following cases:

\begin{itemize}
    \item  If $x_1 \neq 0$, then we directly obtain $a_1^x = b_{1,1}$. For each $i$ satisfying $2 \le i \le n$, the remaining unknowns $a_i^x$ can be  determined by the following explicit formula:
$$a_i^x = b_{i,1} + \frac{1}{x_1}\Bigl( (b_{i,2}-b_i^x)x_2+\sum_{j=3}^i (b_{i,j} - \bigl(a_1^x\bigl)^{j-2} b_{i-j+2}^x) x_j\Bigl),$$
where $2 \le i \le n$, and these equalities give a solution of the system of equations~\eqref{local sysytem m0n}. Here  $b_{2}^x,b_{3}^x, \cdots , b_{n}^x$ are defined arbitrary.[Indeed, since the system of linear equations~\eqref{local sysytem m0n} generally comprises \(n\) equations and \(2n-1\) unknowns, it exhibits \(n-1\) degrees of freedom. Consequently, these particular parameters can be chosen arbitrarily with respect to the remaining variables to guarantee the consistency of the system.]
\item  if $x_1=0$ and $x_2 \neq 0$, then $b_{2,2}=b_2^x$ for each $i$ satisfying $3 \le i \le n$, the remaining unknowns $b_i^x$ can be recursively determined by the following explicit formula:

$$b_i^x = b_{i,2} + \frac{1}{x_2}\Bigl( \sum_{j=3}^i (b_{i,j} - \bigl(a_1^x\bigl)^{j-2} b_{i-j+2}^x) x_j\Bigl),$$
give solution of the system of equation~\eqref{local sysytem m0n}. Here the variables $a_1^x$ is arbitrary value.

\item If $ x_1=x_2=\cdots= x_{m-1} = 0$ and $x_m \neq 0, $ then the following equalities, $m \le i \le n$

$b_2^x=\frac{b_{ii}}{(a_1^x)^{i-2}}$

$b_3^x=\frac{1}{(a_1^x)^{i-2} x_i} \bigl( b_{i+1,i}x_i+b_{i+1,i+1}x_{i+1}-\frac{b_{i,i}}{a_1^x} x_{i+1} \bigl)$

$\cdots\cdots$

$b_{i-m+2}^x=\frac{b_{i,m}}{(a_1^x)^{m-2}}+\frac{1}{(a_1^x)^{m-2}x_m} \sum_{j=m+1}^{i}(b_{i,j}-(a_1^x)^{m-2}b_{i-j+2}^x)x_{m+1}$

give solution of the system of equation~\eqref{local sysytem m0n}. Here the variables $a_1^x$ is arbitrary value.
\end{itemize}
Consequently, the system of linear equations~\eqref{local sysytem m0n} consistently possesses a solution
for every element $x \in m_0(n)$, which implies that the linear operator induced by
the matrix $A$ constitutes a local automorphism. Owing to the fact that the local
automorphism $\nabla$ was chosen arbitrarily and is pointwise defined by the
automorphisms as characterized in Theorem~\eqref{Theoremm0n}, it follows that the matrix
representation of any local automorphism of the algebra $m_0(n)$ necessarily admits
a lower triangular form.
\end{proof}

We can similarly prove the following theorems.

\subsection{Local automorphisms algebra \texorpdfstring{$m_2(n)$}{m(2n)}}\label{subsec10}

The following theorems can be proven similar to the proof of Theorem~\eqref{thm:local-m0}.

\begin{Th}\label{thm:local-m2n}
A linear map $\nabla$ on the filiform Lie algebra $\mathfrak{m}_2(n)$ ($n \ge 5$) is a local automorphism if and only if its matrix representation $M_{\nabla} = (d_{i,j})_{n \times n}$ with respect to the homogeneous basis $\{e_1, e_2, \dots, e_n\}$ is lower triangular and invertible, i.e.,
\begin{equation}\label{eq:local-matrix-form m_2(n)}
M_{\nabla} = \begin{pmatrix}
b_{1,1} & 0 & 0 & 0 & \cdots & 0 & 0 \\
0 & b_{2,2} & 0 & 0 & \cdots & 0 & 0 \\
b_{3,1} & b_{3,2} & b_{3,3} & 0 & \cdots & 0 & 0 \\
b_{4,1} & b_{4,2} & b_{4,3} & b_{4,4} & \cdots & 0 & 0 \\
\vdots & \vdots & \vdots & \vdots & \ddots & \vdots & 0 \\
b_{n-1,1} & b_{n-1,2} & b_{n-1,3} & b_{n-1,4} & \cdots & b_{n-1,n-1} & 0 \\
b_{n,1} & b_{n,2} & b_{n,3} & b_{n,4} & \cdots & b_{n,n-1} & b_{n,n}
\end{pmatrix},
\end{equation}
where $b_{i,j} \in \mathbb{C}$ and $b_{i,i} \neq 0$ for all $1 \le i \le n$.
\end{Th}

\subsection{Local automorphisms of the algebra \texorpdfstring{$W^+(n)$}{W(n)}}\label{subsec11}
\begin{Th}\label{thm:local-W}
A linear map $\nabla$ on the Lie algebra $W^+(n)$ ($n \ge 12$) is a local automorphism if and only if its matrix $M_{\nabla}$ with respect to the basis $\{e_1, e_2, \dots, e_n\}$ has a lower triangular form:
\begin{equation}\label{eq:local-matrix-W-form}
M_{\nabla} = \begin{pmatrix}
b_{1,1} & 0 & 0 & 0 & \cdots & 0 & 0 \\
0 & b_{2,2} & 0 & 0 & \cdots & 0 & 0 \\
b_{3,1} & b_{3,2} & b_{3,3} & 0 & \cdots & 0 & 0 \\
b_{4,1} & b_{4,2} & b_{4,3} & b_{4,4} & \cdots & 0 & 0 \\
\vdots & \vdots & \vdots & \vdots & \ddots & \vdots & 0 \\
b_{n-1,1} & b_{n-1,2} & b_{n-1,3} & b_{n-1,4} & \cdots & b_{n-1,n-1} & 0 \\
b_{n,1} & b_{n,2} & b_{n,3} & b_{n,4} & \cdots & b_{n,n-1} & b_{n,n}
\end{pmatrix},
\end{equation}
where $b_{i,j} \in \mathbb{C}$ and $b_{i,i} \neq 0$ for all $1 \le i \le n$.
\end{Th}

\subsection{Local automorphisms of the algebra \texorpdfstring{$m_{0,1}(2k+1)$}{m(2k+1)}}\label{subsec12}

\begin{Th}\label{thm:local-m01n}
A linear map $\nabla$ on the Lie algebra $m_{0,1}(n)$ ($n = 2k+1, k \ge 3$) is a local automorphism if and only if its matrix $M_{\nabla}$ with respect to the basis $\{e_1, e_2, \dots, e_n\}$ has the following lower triangular form:
\begin{equation}\label{eq:local-matrix-m01n-form}
M_{\nabla} = \begin{pmatrix}
b_{1,1} & 0 & 0 & 0 & \cdots & 0 & 0 \\
b_{2,1} & b_{2,2} & 0 & 0 & \cdots & 0 & 0 \\
b_{3,1} & b_{3,2} & b_{3,3} & 0 & \cdots & 0 & 0 \\
b_{4,1} & b_{4,2} & b_{4,3} & b_{4,4} & \cdots & 0 & 0 \\
\vdots & \vdots & \vdots & \vdots & \ddots & \vdots & 0 \\
b_{n-1,1} & b_{n-1,2} & b_{n-1,3} & b_{n-1,4} & \cdots & b_{n-1,n-1} & 0 \\
b_{n,1} & b_{n,2} & b_{n,3} & b_{n,4} & \cdots & b_{n,n-1} & b_{n,n}
\end{pmatrix},
\end{equation}
where $d_{i,j} \in \mathbb{C}$ .
\end{Th}

\subsection{Local automorphisms of the algebra \texorpdfstring{$m_{0,2}(2k+2)$}{m(2k+2)}}\label{subsec13}

\begin{Th}\label{thm:local-m02n}
A linear map $\nabla$ on the Lie algebra $m_{0,2}(n)$ ($n = 2k+2, k \ge 3$) is a local automorphism if and only if its matrix $M_{\nabla}$ with respect to the basis $\{e_1, e_2, \dots, e_n\}$ has the following lower triangular form:
\begin{equation}\label{eq:local-matrix-m02n-form}
M_{\nabla} = \begin{pmatrix}
b_{1,1} & 0 & 0 & 0 & \cdots & 0 & 0 \\
0 & b_{2,2} & 0 & 0 & \cdots & 0 & 0 \\
b_{3,1} & b_{3,2} & b_{3,3} & 0 & \cdots & 0 & 0 \\
b_{4,1} & b_{4,2} & b_{4,3} & b_{4,4} & \cdots & 0 & 0 \\
\vdots & \vdots & \vdots & \vdots & \ddots & \vdots & 0 \\
b_{n-1,1} & b_{n-1,2} & b_{n-1,3} & b_{n-1,4} & \cdots & b_{n-1,n-1} & 0 \\
b_{n,1} & b_{n,2} & b_{n,3} & b_{n,4} & \cdots & b_{n,n-1} & b_{n,n}
\end{pmatrix},
\end{equation}
where $b_{i,j} \in \mathbb{C}$ with the non-vanishing diagonal constraints $b_{1,1} \neq 0$ and $b_{2,2} = b_{1,1}^2$.
\end{Th}

\begin{proof}
By Theorem~\eqref{thm:autm(2k+2)}  the automorphism $\phi_x$ has the following matrix form:

\begin{equation}
A_{x} =
\begin{pmatrix}\label{matrix local form m02n}
a_1^x & 0 & 0 & 0 & 0 & \cdots & 0 & 0 \\
0 & (a_1^x)^2 & 0 & 0 & 0 & \cdots & 0 & 0 \\
a_3^x & b_3^x & (a_1^x)^3 & 0 & 0 & \cdots & 0 & 0 \\
a_4^x & b_4^x & a_1^x b_3^x & (a_1^x)^4 & 0 & \cdots & 0 & 0 \\
a_5^x & b_5^x & a_1^x b_4^x & (a_1^x)^2 b_3^x & (a_1^x)^5 & \cdots & 0 & 0 \\
\vdots & \vdots & \vdots & \vdots & \vdots & \ddots & \vdots & \vdots \\
a_{n-1}^x & b_{n-1}^x & c^x(n-1,3) & c^x(n-1,4) & c^x(n-1,5) & \cdots & (a_1^x)^{n-1} & 0 \\
a_n^x & b_n^x & c^x(n,3) & c^x(n,4) & c^x(n,5) & \cdots & c_{n,n-1}^x & (a_1^x)^n
\end{pmatrix}
\end{equation}

in here
$c^x_{n-1,t}=(a_1^x)^{t-2}b^x_{n-t+1}+(a_1^{\,t-3})^x\sum_{l=2}^{(n-1)/2} (-1)^{l+1} (a^x_l b^x_{n-t-l+2} - a^x_{n-l-1} b^x_{l-t+3})$
and

$c^x_{n,t}=(a_1^x)^{t-2}b_{n-t+2}+(a_1^x)^{\,t-3}\sum_{p=2}^{(n-1)/2} (-1)^{p+1}(\frac{n}{2}-p)  (a^x_p b^x_{n-p-t+3} - a^x_{n-p} b^x_{p-t+3})$

The present proof is analogous to that of Theorem~\eqref{thm:local-m01n}, the only additional condition being that \(a_2^x = 0\). Matrix~\eqref{matrix local form m02n}, meaning that the coefficients are defined in essentially the same manner.

\begin{equation}\label{local sysytem m02n}
    \begin{cases}
    b_{1,1}x_1=a_1^x x_1 \\
    b_{2,2}x_2=(a_1^x)^2x_2 \\
     \sum_{j=1}^{3}b_{3,j}x_j = a_{3}^x x_1 + b_3^x x_2 +  (a_{1}^x)^3 x_3 , \\
    \sum_{j=1}^{i}b_{i,j}x_j=a_i^x x_1+b_i^x x_2+\sum_{j=3}^{i-1}\bigl(a_1^x\bigl)^{j-2} b_{i-j+2}^xx_j+(a_1^x)^i x_i, \quad 4 \le i\le n-2 \\
    \sum_{j=1}^{n-1}b_{n-1,j}x_j=a_{n-1}^x x_1+b_{n-1}^x x_2+\sum_{j=3}^{n-2}c_{n-1,j}^xx_j+(a_1^x)^{n-1} x_{n-1} \\
    \sum_{j=1}^{n}b_{n,j}x_j=a_{n}^x x_1+b_{n}^x x_2+\sum_{j=3}^{n-1}c_{n,j}^xx_j+(a_1^x)^{n} x_{n}
    \end{cases}
\end{equation}

\begin{itemize}
    \item  If $x_1 \neq 0$, then $a_1^x = b_{1,1}$.
    $$a_{n-1}^x = \frac{1}{x_1}\Bigl( \sum_{j=1}^{n-1}b_{n-1,j}x_j-b_{n-1}^x x_2-\sum_{j=3}^{n-2}c_{n-1,j}^xx_j-(b_{1,1})^n x_{n-1}\Bigl),$$
$$a_n^x = \frac{1}{x_1}\Bigl( \sum_{j=1}^{n}b_{n,j}x_j-b_n^x x_2-\sum_{j=3}^{n-1}c_{n,j}^xx_j-(b_{1,1})^n x_n\Bigl),$$
$b_k^x $ is arbitrary

so we can write if $a_k^x $ is arbitrary:
 $$b_{n-1}^x = \frac{1}{x_2}\Bigl( \sum_{j=1}^{n-1}b_{n-1,j}x_j-a_{n-1}^x x_1-\sum_{j=3}^{n-2}c_{n-1,j}^xx_j-(b_{1,1})^n x_{n-1}\Bigl),$$
$$b_n^x = \frac{1}{x_2}\Bigl( \sum_{j=1}^{n}b_{n,j}x_j-a_n^x x_1-\sum_{j=3}^{n-1}c_{n,j}^xx_j-(b_{1,1})^n x_n\Bigl),$$

\item  if $x_1=0$ and $x_2 \neq 0$, then $b_{2,2}=(a_1^x)^2 \longrightarrow a_1^x=\sqrt{b_{2,2}}$ .
$$b_{n-1}^x =\frac{1}{x_2}\Bigl( \sum_{j=1}^{n-1}b_{n-1,j}x_j-\sum_{j=3}^{n-2}c_{n-1,j}^xx_j-(\sqrt{b_{2,2}})^{n-1} x_{n-1}\Bigl),$$

$$b_n^x =\frac{1}{x_2}\Bigl( \sum_{j=1}^{n}b_{n,j}x_j-\sum_{j=3}^{n-1}c_{n,j}^xx_j-(\sqrt{b_{2,2}})^n x_n\Bigl),$$
give solution of the system of equation~\eqref{local sysytem m02n}. Here the variables $a_i^x$ is arbitrary value.

\item If $ x_1=x_2=\cdots= x_{m-1} = 0$ and $x_m \neq 0, $ then the following equalities, $m \le i \le n$

$b_{m,m}={(a_1^x)^{m}} \Longrightarrow a_1^x=\sqrt[m]{b_{m,m}}$

$b^x_{n-m+1}= \frac{1}{a_1^{\,m-2} x_m} \bigl(\sum_{j=m}^{n-1}b_{i,j} -a_1^{\,m-3}\sum_{l=2}^{(n-2)/2} (-1)^{l+1} (a^x_l b^x_{n-m-l+2} - a^x_{n-l-1} b^x_{l-m+3}) x_m- \\ \sum_{j=m+1}^{n-2}c_{n-1,j}^xx_j+(\sqrt[m]{b_{m,m}})^{n-1} x_{n-1} \bigl)$

$b^x_{n-m+2}= \frac{1}{a_1^{\,m-2} x_m} \bigl(\sum_{j=m}^{n}b_{i,j} -(a_1^x)^{\,m-3}\sum_{p=2}^{(n-1)/2} (-1)^{p+1}(\frac{n}{2}-p)  (a^x_p b^x_{n-p-m+3} -  a^x_{n-p} b^x_{p-m+3})x_m- \\ \sum_{j=m+1}^{n-1}c_{n,j}^xx_j+(\sqrt[m]{b_{m,m}})^n x_n \bigl)$

give solution of the system of equation~\eqref{local sysytem m02n}. Here the variables $a_i^x$ is arbitrary value.
\end{itemize}

\end{proof}

We can similarly prove the following theorem.

\subsection{Local automorphisms of the algebra \texorpdfstring{$m_{0,3}(2k+3)$}{m(2k+3)}}\label{subsec14}
In this case, the verification for local automorphisms is very similar to that for  $m_{0,2}(2k+2)$. In both cases, the conditions  $a_2^x = 0 $ and  $b_2^x = (a_1^x)^2$ are common, and the entries of the first n-3 rows coincide.

\begin{Th}\label{thm:local-m03n}
A linear map $\nabla$ on the Lie algebra $m_{0,3}(n)$ ($n = 2k+3, k \ge 3$) is a local automorphism if and only if its matrix $M_{\nabla}$ with respect to the basis $\{e_1, e_2, \dots, e_n\}$ has the following lower triangular form:
\begin{equation}\label{eq:local-matrix-m03n-form}
M_{\nabla} = \begin{pmatrix}
b_{1,1} & 0 & 0 & 0 & \cdots & 0 & 0 \\
0 & b_{2,2} & 0 & 0 & \cdots & 0 & 0 \\
b_{3,1} & b_{3,2} & b_{3,3} & 0 & \cdots & 0 & 0 \\
b_{4,1} & b_{4,2} & b_{4,3} & b_{4,4} & \cdots & 0 & 0 \\
\vdots & \vdots & \vdots & \vdots & \ddots & \vdots & 0 \\
b_{n-1,1} & b_{n-1,2} & b_{n-1,3} & b_{n-1,4} & \cdots & b_{n-1,n-1} & 0 \\
b_{n,1} & b_{n,2} & b_{n,3} & b_{n,4} & \cdots & b_{n,n-1} & b_{n,n}
\end{pmatrix},
\end{equation}
where $b_{i,j} \in \mathbb{C}$ with the non-vanishing diagonal constraints $b_{1,1} \neq 0$ and $b_{2,2} = b_{1,1}^2$.
\end{Th}

\section{2-local automorphism}\label{sec5}

The present section is devoted to the description of $2$-local automorphisms of totally graded filiform Lie algebras.

Let $\mathfrak{A}$ be an algebra. A (not necessarily linear) map $\Delta : \mathfrak{A} \to \mathfrak{A}$ is called a $2$-local automorphism if, for any elements $x, y \in \mathfrak{A}$, there exists an automorphism $\varphi_{x,y} : \mathfrak{A} \to \mathfrak{A}$ such that
\[
\Delta(x) = \varphi_{x,y}(x), \qquad \Delta(y) = \varphi_{x,y}(y).
\]

\begin{Th}
The Lie algebra $\mathfrak{A}$ in Tables~\eqref{tab:filiformn},~\eqref{tab:filiform}  has 2-local automorphisms which are not automorphisms.
\end{Th}

\begin{proof}
Let us define a homogeneous non-additive function $f$ on $\mathbb{C}^2$ as follows:
\begin{equation*}
f(x_1, x_2) = \begin{cases}
\dfrac{x_1^3}{x_1^2 + x_2^2}, & \text{if } x_1^2 + x_2^2 \neq 0, \\
0, & \text{if } x_1^2 + x_2^2 = 0,
\end{cases}
\end{equation*}
where $(x_1, x_2) \in \mathbb{C}^2$. Consider the map $\Delta : \mathfrak{A} \to \mathfrak{A}$ defined in the following way:
\begin{equation*}
\Delta(x) = x + f(x_1, x_2)e_n, \quad \text{where } x = \sum_{i=1}^n x_i e_i \in \mathfrak{A},
\end{equation*}
and $\{e_1, e_2, \dots, e_n\}$ is the basis of $\mathfrak{A}$. Since $f$ is not additive, we have that $\Delta$ is not an automorphism of $\mathfrak{A}$.

Let us show that $\Delta$ is a 2-local automorphism. Let $v, w$ be arbitrary elements of $\mathfrak{A}$. We take a linear map $\Phi : \mathfrak{A} \to \mathfrak{A}$ with a matrix of the following form with respect to the basis $\{e_1, e_2, \dots, e_n\}$:
\begin{equation*}
\Phi = \begin{pmatrix}
1 & 0 & 0 & 0 & \cdots & 0 & 0 \\
0 & 1 & 0 & 0 & \cdots & 0 & 0 \\
0 & 0 & 1 & 0 & \cdots & 0 & 0 \\
0 & 0 & 0 & 1 & \cdots & 0 & 0 \\
\vdots & \vdots & \vdots & \vdots & \ddots & \vdots & \vdots \\
0 & 0 & 0 & 0 & \cdots & 1 & 0 \\
a & b & 0 & 0 & \cdots & 0 & 1
\end{pmatrix}.
\end{equation*}
By the definition of the automorphism matrix $A_{\mathfrak{A}}$ , $\Phi$ is a  automorphism of $\mathfrak{A}$. We prove that there exist parameters $a$ and $b$ such that:
\begin{equation*}
\Delta(v) = \Phi(v), \quad \Delta(w) = \Phi(w).
\end{equation*}
Indeed, the last two vector equations give us the following system of algebraic equations with respect to the variables $a$ and $b$:
\begin{equation}\label{eq:system-A}
\begin{cases}
a v_1 + b v_2 = f(v_1, v_2) \\
a w_1 + b w_2 = f(w_1, w_2)
\end{cases}
\end{equation}
Clearly, if $v_1 w_2 - v_2 w_1 = 0$, then the vectors $(v_1, v_2)$ and $(w_1, w_2)$ are proportional. Since the right-hand side of the system \eqref{eq:system-A} is structurally homogeneous of degree 1 (i.e., $f(\lambda z_1, \lambda z_2) = \lambda f(z_1, z_2)$), the equations become scalar multiples of each other, and the system \eqref{eq:system-A} has infinitely many solutions.

Else, if $v_1 w_2 - v_2 w_1 \neq 0$, the determinant of the system matrix is non-zero, and the system \eqref{eq:system-A} has a unique solution for $a$ and $b$.

Thus, for any $v, w$ in $\mathfrak{A}$, there exists a  linear automorphism $\Phi$ of $\mathfrak{A}$ such that $\Delta(v) = \Phi(v)$ and $\Delta(w) = \Phi(w)$, i.e., $\Delta$ is a  2-local automorphism on $\mathfrak{A}$.
\end{proof}

\end{document}